\documentclass[11pt]{amsart}
\usepackage[mathscr]{eucal}
\usepackage{palatino, mathpazo, amsfonts, mathrsfs, amscd}
\usepackage[all]{xy}
\usepackage{url}
\usepackage{amssymb, amsmath, amsthm}
\usepackage{stackrel}
\usepackage{bm}

\usepackage{tikz-cd}

\newtheorem{theorem}{Theorem}[section]
\newtheorem{lemma}[theorem]{Lemma}

\newtheorem{proposition}[theorem]{Proposition}
\newtheorem{corollary}[theorem]{Corollary}
\newtheorem{conjecture}[theorem]{Conjecture}
\theoremstyle{remark}
\newtheorem{remark}[theorem]{Remark}

\newtheorem{definition}[theorem]{Definition}
\newtheorem{notation}[theorem]{Notation}

\numberwithin{equation}{subsection}
\usepackage{todonotes}

\makeatletter
\def\imod#1{\allowbreak\mkern10mu({\operator@font mod}\,\,#1)}
\makeatother

\newcommand{\on}{\operatorname}

\newcommand{\cY}{\mathcal{Y}}
\newcommand{\cC}{\mathcal{C}}
\newcommand{\cX}{\mathcal{X}}
\newcommand{\cH}{\mathcal{H}}
\newcommand{\cL}{\mathcal{L}}
\newcommand{\cF}{\mathcal{F}}
\newcommand{\cD}{\mathcal{D}}
\newcommand{\cZ}{\mathcal{Z}}

\newcommand{\CC}{\mathbb{C}}
\newcommand{\ZZ}{\mathbb{Z}}

\newcommand{\PP}{\mathbb{P}}
\newcommand{\QQ}{\mathbb{Q}}
\newcommand{\UU}{\mathbb{U}}
\newcommand{\VV}{\mathbb{V}}

\newcommand{\sL}{\mathscr{L}}

\newcommand{\sO}{\mathscr{O}}

\newcommand{\sV}{\mathscr{V}}
\newcommand{\sW}{\mathscr{W}}

\newcommand{\sMbar}{\overline{\mathscr{M}}}
\newcommand{\sAbar}{\overline{\mathscr{A}}}

\newcommand{\bq}{\mathbf{q}}
 \newcommand{\bp}{\mathbf{p}}
  \newcommand{\bo}{\mathbf{0}}
\newcommand{\bt}{\mathbf{t}}
\newcommand{\bs}{\mathbf{s}}
\newcommand{\ii}{\mathbb{1}}
\newcommand{\jj}{\mathfrak{j}}
\newcommand{\mm}{\mathtt{m}}
\newcommand{\h}{h}
\newcommand{\g}{g}
\newcommand{\bv}[1]{\mathbf{#1}}
\newcommand{\mlk}{MLK }

\DeclareMathOperator{\Gr}{Gr}
\DeclareMathOperator{\age}{age}
\DeclareMathOperator{\ch}{ch}

\DeclareMathOperator{\Res}{Res}

\newcommand{\Sing}{\mathit{Sing}}

\newcommand{\br}[1]{\left\langle#1\right\rangle}  
\newcommand{\set}[1]{\left\{#1\right\}}  

\makeatletter
\def\Ddots{\mathinner{\mkern1mu\raise\p@
\vbox{\kern7\p@\hbox{.}}\mkern2mu
\raise4\p@\hbox{.}\mkern2mu\raise7\p@\hbox{.}\mkern1mu}}
\makeatother

\author{Y.-P.~Lee}
\address{Y.-P.~Lee, Department of Mathematics, University of Utah,
Salt Lake City, Utah 84112-0090, U.S.A.}
\email{yplee@math.utah.edu}

\author{N.~Priddis}
\address{N.~Priddis, Institut of Algebraic Geometry, 
Gottfried Wilhelm Leibniz Universit\"at Hannover, 
Welfengarten 1, 30167 Hannover, Germany}
\email{priddis@math.uni-hannover.de}

\author{M.~Shoemaker}
\address{M.~Shoemaker, Department of Mathematics, University of Utah,
Salt Lake City, Utah 84112-0090, U.S.A.}
\email{markshoe@math.utah.edu}

\subjclass{2010 Mathematics Subject Classification: 14N35, 14E30}

\title[CTC $\Rightarrow$ LG/CY]{A proof of \\ the Landau--Ginzburg/Calabi--Yau correspondence \\ via the crepant transformation conjecture}

\begin{document}

\begin{abstract}
We establish a new relationship (the \mlk correspondence)
between twisted FJRW theory and local Gromov--Witten theory in all genera.
As a consequence, we show that the Landau--Ginzburg/Calabi--Yau correspondence
is implied by the crepant transformation conjecture for Fermat type 
in genus zero.
We use this to then prove the  Landau--Ginzburg/Calabi--Yau correspondence for Fermat
type, generalizing the results of A.~Chiodo and Y.~Ruan in \cite{ChR}.
\end{abstract}

\maketitle

{\small
\tableofcontents}

\section{Introduction}
The crepant transformation conjecture describes a relationship between
the Gromov--Witten theories of $K$-equivalent varieties in terms of analytic continuation
and symplectic transformation.  A more recent conjecture, the Landau--Ginzburg/Calabi--Yau (LG/CY)
correspondence, proposes a similar relationship between the Gromov--Witten theory of
a Calabi--Yau variety and the FJRW theory of a singularity.  The primary goal of this paper is 
to relate these two conjectures.

FJRW theory was constructed by 
Fan, Jarvis and Ruan (\cite{FJR1}) 
as a  ``Landau--Ginzburg (LG) $A$ model''
to verify a conjecture of Witten \cite{W1}.  
The construction gives a cohomological field theory
defined by a virtual class on a cover of the moduli space of curves.
It may be viewed as an analogue of Gromov--Witten theory, yielding invariants of
 a singularity rather than a smooth variety.
Roughly, the input of the theory is an \emph{LG pair} $(Q, G)$ where $Q$ is a 
quasi-homogeneous polynomial $Q: \CC^N \to \CC$, 
and $G$ an \emph{admissible group} of diagonal automorphisms of $Q$
(See Section~\ref{ss:Wstr}).
The moduli space is defined to be $N$-tuples of line bundles on curves,
$\cL_i \to \cC$, such that $Q_s(\cL_1, ..., \cL_N) \cong \omega_{\cC, \log}$, 
where $Q_s$ denotes a monomial of $Q$ and $\omega_{\cC, \log}$ is
the log-canonical bundle.
The most difficult part of the construction is to define the
virtual classes. This was done in the analytic category by Fan--Jarvis--Ruan 
in \cite{FJR1} and in the algebraic category by Polishchuk--Vaintrob 
in \cite{PV}.

The first proof of the LG/CY correspondence was given by Chiodo and Ruan for the quintic threefold (\cite{ChR}).
When $Q$ is the Fermat quintic in five variables and $G = \on{diag}(\mu_5)$, they
 proved  that the 
genus zero FJRW theory of $(Q, G)$
is equivalent to genus zero Gromov--Witten theory of the quintic hypersurface $Z(Q) = \{Q = 0\}$ in $\PP^4$.   
The identification of the two theories is given by analytic continuation
and 
symplectic transformation by an element Givental's symplectic loop group (\cite{G4, YPnotes}).


The LG/CY correspondence has now been proven in genus zero in
 all cases where the Calabi--Yau is a hypersurface in projective space 
 (\cite{CIR}) as well as for the mirror quintic (\cite{PS}).
There are two aspects of
previous proofs however which, in our opinion, warrant further investigation.
%
First, the proofs of the LG/CY correspondence to-date have been
computational in nature, and do not explain the source behind this correspondence.
%
%
%
Second, due to the existence of stabilizers in the action
of the symplectic loop group, relating the genus zero FJRW and GW theory in this 
correspondence requires one to make a choice of
symplectic transformation.
Crucially, two symplectic transformations which have the 
same effects on the genus zero theory might have quantizations which 
act differently on 
higher genus theories.  Therefore, any correspondence in higher genus
requires a canonical way of choosing the symplectic transformation relating the genus zero invariants.
%
%

The goal of the present paper is to help elucidate the questions raised above by 
proposing a more conceptual framework for the LG/CY correspondence.  Namely we 
relate it to the older and better understood crepant transformation conjecture.


We start with the observation that in the moduli problem
for FJRW theory, one may replace the log-canonical
bundle with any power of the log-canonical bundle.  This yields an isomorphic
moduli space, and for any given power of the log-canonical bundle
one may construct a corresponding cohomological field theory.
%
If this power is the zeroth power, i.e., the trivial line bundle,
then one recovers the orbifold GW theory for the abelian quotient
stack $[\mathbb{C}^N/{G}]$.

As a first step, we restrict ourselves to LG pairs $(Q, G)$
where $Q$ is a quasi-homogeneous polynomial of Fermat type
i.e., $Q =\sum_{i=1}^N x_i^{d/c_i}$. 
In this case
we prove a new correspondence
 (dubbed the ``multiple log-canonical'' or \mlk correspondence, see Section~\ref{s:5.2.3}) 
 which describes the genus zero FJRW theory
of $(Q,G)$ in terms of the genus zero orbifold GW theory of $[\mathbb{C}^N/{G}]$.

Via the \mlk correspondence, we prove that the LG/CY correspondence
can be deduced from the crepant transformation conjecture (CTC).  Symbolically we 
write:
\[
 \text{CTC $\Rightarrow$ LG/CY.}
\] 

More precisely, let $\PP(G) := [\PP(c_1, ...,c_N)/\bar{G}]$, where
$\bar{G}$ is the quotient of $G$ by those elements acting trivially on $[\PP(c_1, ...,c_N)]$.
Let $K_{\PP(G)}$ denote the total space of the canonical bundle over $\PP(G)$ and let
$Z(Q) \subset \PP(G)$ be the Calabi--Yau orbifold defined by $Q$.
Then the LG/CY correspondence may be established by a special
case of the CTC.
The relationship is summarized in the following diagram.
\[
\begin{tikzcd}
 GW_0 (K_{P(G)}) \arrow[rightarrow]{r}{\text{QSD}} \arrow[leftrightarrow]{d}{\text{CTC}} &GW_0 (Z(Q)) \arrow[leftrightarrow]{d}{\text{LG/CY}} \\
 GW_0 ([\mathbb{C}^N /G]) \arrow[rightarrow]{r}{\text{\mlk}} &\text{FJRW}_0(Q, G)
\end{tikzcd}
\]
In the upper right corner is the genus zero GW theory for the Calabi--Yau orbifold
Z(Q).
The lower right corner is the genus zero FJRW theory associated to the LG pair $(Q, G)$.
The right vertical arrow is the LG/CY correspondence discussed above.
The left vertical arrow is the CTC relating the genus zero orbifold GW theory
of $[\mathbb{C}^N/G]$ with the genus zero GW theory of its crepant partial resolution $K_{\PP(G)}$.
The upper horizontal arrow is quantum Serre duality (\cite{CG}), which relates
the GW theory of the total space of a line bundle with the GW theory
of the hypersurface defined by a section of the line bundle.
The \mlk correspondence, established in Section~\ref{s:5.2.3}, completes the square.

The upshot of this approach is that both the QSD and \mlk correspondences
take a relatively simple form.  Thus the complicated form of the LG/CY correspondence (e.g., analytic continuation, symplectic transformation) 
may be understood directly from the 
crepant transformation conjecture.  In particular, in the statement of the crepant transformation conjecture (\cite{CIT}),
the form of the symplectic transformation is subject to several constrains, thereby limiting
the choices which can be made, and partially addressing the non-canonical nature of the
symplectic transformation in previous proofs of the LG/CY correspondence.

\subsection{Contents of the paper}
In Section~\ref{s:1} we give a general construction of a cohomological field theory defined 
as a twisted theory over a generalization of the moduli of $r$-spin curves.
In Section~\ref{s:rtoi} we show how in special cases of the above construction
one recovers
the Gromov--Witten theory of local affine quotients as well as the genus zero
FJRW theory of Fermat LG pairs.  
Section~\ref{ss:sympform} gives a brief summary of Givental's symplectic formalism
which we use in Section~\ref{s:correspondences} to compute the
cohomological field theories introduced earlier.
In Section~\ref{s:cc2} we are able to state and prove the 
\mlk correspondence, which relates 
the genus zero Gromov--Witten theory of affine quotients to the
FJRW theory of Fermat LG pairs.
We then apply this correspondence in Section~\ref{s:ctclgcy} to show that the 
crepant transformation conjecture implies the LG/CY correspondence in 
a large class of cases.
Finally in Section~\ref{s:app} we prove a version of the LG/CY correspondence for 
the cases of interest to us.

%

\subsection{Acknowledgments}
The authors would like to thank T. Coates, H. Iritani, and Y. Jiang for useful conversations and for
providing them with an early copy of their paper ``The crepant transformation conjecture for toric complete intersections'' (\cite{CIJ}).  They are also grateful to Y. Ruan for 
helping explain FJRW theory and for providing much of the initial motivation for this project. 
Y.-P.~L.\ was partially supported by the NSF.
N.~P.\ was partially supported by the NSF grant RTG 1045119.  
M.~S.\ was partially supported by NSF RTG Grant DMS-1246989.

\section{Twisted invariants}\label{s:1}

\subsection{$W$ structures}\label{ss:Wstr}

\begin{definition}[{\cite[Definition~A.1]{ChR}}]
Let $d$ be a non-negative integer.
A \emph{$d$-stable $n$-pointed genus $h$ curve} is an $n$-pointed stable 
orbi-curve 
such that all marked points and nodes have cyclic stabilizers of order $d$
and no other non-trivial stabilizers.
\end{definition}

\begin{notation}
Let $\sMbar_{\h, n}^{d}$ denote the moduli space of $d$-stable $n$-pointed genus 
$h$ curves. 
A $d$-stable curve (or a family of such) 
will always be denoted by $\cC$.  
Let $\sAbar_{\h, n}^{(d, c)}$ denote the moduli space of $d$-th roots of 
$\omega_{\cC/\sMbar, \log}^{\otimes c}$.  
\end{notation}


Let $Q: \CC^N \to \CC$ be a nondegenerate quasi-homogeneous polynomial, 
i.e., for $\alpha \in \CC^*$,
\[
 Q( \alpha^{c_1} x_1, \ldots, \alpha^{c_N}x_N) = \alpha^d Q(x_1, \ldots, x_N),
\] 
where the $c_j$'s are positive integers.  We assume always that
$\gcd(c_1, \ldots, c_N) = 1$.
$Q$ is said to have degree $d$ with integer weights $c_1, \ldots, c_N$.

Let 
$$ G_Q = G^{\on{max}}_Q :=  (\CC^*)^N \cap \operatorname{Aut}(Q)$$ 
denote the (maximal) group of diagonal automorphisms of $Q$.  
We define a distinguished element $\jj \in G_Q$, the \emph{grading} element,
by
\[ 
 \jj := \big(\exp\big(2 \pi i \frac{c_1}{d}\big), \ldots, \exp\big(2 \pi i \frac{c_N}{d}\big)\big).
\]  
Let $\bar d$ denote the \emph{period} of $G_Q$, defined as 
$$\bar{d} := \max\left\{|g|\, \big| g \in G \right\},$$ 
and let $\bar c_j = c_j \bar d/d$.  Then we may write 
$$\jj = \big(\exp\big(2 \pi i \frac{\bar c_1}{\bar d}\big), \ldots, 
  \exp\big(2 \pi i \frac{\bar c_N}{\bar d}\big)\big),$$ 
which will be convenient since we will work on $\bar d$-stable curves.

\begin{definition}
On a marked $\bar d$-stable curve $\cC$, a \emph{$W^c$ structure} is the data of $N$ $\bar d$-th roots of the log canonical bundle
\[
  (\cL_j, \phi_j: \cL_j^{\otimes \bar d} 
  \stackrel{\cong}{\to} \omega_{\cC, \log}^{\otimes (c \cdot \bar c_j)})
\]
which satisfy
\begin{equation}  \label{e:cond1}
  Q_s (\cL_1, \ldots, \cL_N) \cong \omega_{\cC, \log}^{\otimes c}
\end{equation}
for each monomial $Q_s$ in $Q$.
\end{definition}

\begin{remark} The ``$W$'' in $W^c$ structures stands for E.~Witten, 
whose ideas initiated the study of such moduli (\cite{W1}).
\end{remark}

\begin{definition}
The \emph{moduli space $W^c_{\h,n}$ of $W^c$ structures} of $Q$ is the open and 
closed substack of the fiber product 
\[
 \sAbar_{\h, n}^{(\bar d, c \cdot \bar c_1)} \times_{\sMbar_{\h, n}^{\bar d}}
 \cdots \times_{\sMbar_{\h, n}^{\bar d}} \sAbar_{\h, n}^{( \bar d, c \cdot \bar c_N)}
\] 
consisting of those $N$-tuples 
$$(\cL_j, \phi_j: \cL_j^{\otimes \bar d} \stackrel{\cong}{\to} 
 \omega_{\cC, \log}^{\otimes (c \cdot \bar{c_j})} )$$ 
which satisfy \eqref{e:cond1} for all monomials $Q_s$ in $Q$.
\end{definition}

We now add the information of a group of automorphisms into the 
definition of our moduli space.
A group $G \leq G_Q$ is \emph{admissible} if $\jj\in G$. (See \cite[Definition~2.3.2 and Proposition~2.3.5]{FJR1} for an alternative equivalent definition.)
\begin{definition}
A \emph{(gauged) Landau--Ginzburg (LG) pair} is a pair $(Q, G)$ where $Q$ is a nondegenerate quasi-homogeneous polynomial and $G$ is an admissible subgroup of $G_Q$.
\end{definition}

\begin{notation}
Given $g \in G$, let $m_j(g)$ denote the \emph{multiplicity} of $g$ on the 
$j$th factor of $\CC^N$.  
In other words, $g$ acts on $\CC^N$ via $G \subset (\CC^*)^N$ by
$$ \left(\exp(2 \pi i m_1(g)), \ldots, \exp(2 \pi i m_N(g)) \right)$$
such that $0 \leq m_j(g) < 1$.
Let $N_g :=\dim (\CC^N)^g = \#\{j| m_j(\g) = 0\}$.

\end{notation}

Let $Q '$ denote a degree $d$ Laurent polynomial with different monomials than $Q$ such that the group $G_{Q + Q'}$ of diagonal automorphisms of $Q+Q '$ is exactly $G$.  

\begin{definition}\label{d:Wstr} 
The \emph{moduli space $W_{\h, n, G}^c$ of $W^c$ structures of $(Q,G)$} is defined to be the moduli space of $W^c$ structures of $Q + Q'$, where $Q'$ is as above.
\end{definition}  

It is easy to show such $Q'$ exists and that the above definition does not depend on a choice of $Q'$. 

As a consequence of the definition of $W^c_{\h, n, G}$, at each marked point $p_i$, the isotropy acts on fibers of $\oplus_{j=1}^N \cL_j$ by an element of $G$.  
One may therefore break $W^c_{\h, n, G}$ into open and closed substacks based on the action of the corresponding isotropy group.  Let 
\[  W^c_{\h, n, G}(g_1, \ldots, g_n)  \] 
denote the substack where the isotropy at $p_i$ acts by $g_i$.  
The following fact will be used later.

\begin{lemma}[\cite{FJR1}] 
Assume $n > 0$, the stack $W^c_{\h, n, G}$ splits into a disjoint union of 
open and closed substacks
\[
  W^c_{\h, n, G}= \coprod_{g_1, \ldots, g_n \in G} W^c_{\h, n, G}(g_1, \ldots, g_n).
\]  
Furthermore $W^c_{\h, n, G}(g_1, \ldots, g_n)$ is nonempty if and only if 

\begin{equation}\label{e:numcon}
\frac{c c_j}{d}(2h - 2 + n) - \sum_{i=1}^n m_j(g_i) \in  \ZZ \hspace{1 cm} 1 \leq j \leq N
\end{equation}
\end{lemma}

The first statement is easy to see. 
The second statement is essentially proven in \cite[Proposition~2.2.8]{FJR1}.
The numerical condition \eqref{e:numcon} is established using the observation that
the corresponding line bundles $| \cL_j |$ on the coarse moduli have integral degrees, plus 
the calculation 
\begin{equation}\label{e:degree}
 \deg (| \cL_j |) = \frac{c c_j}{d} (2h - 2 + n) - \sum_{i=1}^n m_j(g_i).
\end{equation}

\subsection{``Untwisted'' theories}

There is a map 
$$W^c_{\h, n, G} \to \sMbar_{\h, n}$$ 
obtained by forgetting the line bundles $\cL_j$ as well as the orbifold structure of the underlying curve.  By pulling back $\psi$-classes from $\sMbar_{\h, n}$ we obtain tautological classes on $W^c_{\h, n, G}$.  We can integrate these classes over the moduli space to obtain invariants.

Given a $W^c$ structure of $(Q,G)$, we introduce the \emph{$W^c$ state space} 
as a vector space formally generated by basis vectors $\phi^c_{g}$ for each $g \in G$,
\[
  H^c := \oplus_{g \in G} \CC  \phi^c_{g} .
\]
Define the \emph{untwisted $W^c$ invariant}
\begin{equation} \label{e:1.0.3} 
  \br{ \psi^{a_1} \phi^c_{ g_1}, \ldots, \psi^{a_n} \phi^c_{  g_n} }^c_{\h,n} :=
  \int_{W^c_{\h, n, G}(g_1 \jj^{c},  \ldots, g_n\jj^{c})} \prod_{i = 1}^n \psi_i^{a_i}.
\end{equation}
Note the shifting by $\jj^{c}$ in this definition.
There is a pairing given by
\[ 
 \langle \phi^c_{ g_1}, \phi^c_{ g_2} \rangle^c := \br{ \phi^c_{ g_1}, \phi^c_{ g_2}, \phi^c_{ e}}^c_{0, 3}
\]
where $e$ is the identity element in $G$.

Although the definition of $H^c$ looks somewhat contrived, the corresponding invariants should not.
As it stands, $H^c$ should be viewed as giving ``place-holders'' for the various connected components of the moduli space.  The geometric meaning 
will be clear after we establish the relationship to Gromov--Witten and FJRW theory.

\subsection{Twisted theories}

Let $\CC^*$ act on a $W^c$ structure by acting
on each line bundle.  This induces an action on $W^c_{\h, n , G}$. 
\begin{notation} Let $\lambda$ denote the equivariant parameter, and let $-\lambda_j$ denote the character of the action on the $j$th bundle (i.e. $\lambda_j$ is a multiple of $\lambda$).  We assume always that each character is nontrivial.
\end{notation}

We may express an invertible multiplicative characteristic class as 
\[ 
  \bs: \oplus \cL_j \mapsto \exp\left( \sum_{j = 1}^N \sum_{k \ge 0} s^j_k \ch_k(\cL_j)\right),
\]
where 
$$\exp(s^j_0),\; s^j_k \in \CC[\lambda, \lambda^{-1}] \:\: \text{for} \:\: 1 \leq j \leq N, \:k > 0.$$
We define the \emph{$\bs$-twisted virtual class} on $W^c_{\h, n, G}$ as the class 
$$ [W^c_{\h, n, G}]^{\bs} := 
 \bs(R \pi_* \oplus_{j=1}^N \cL_{j})\cap [W^c_{\h, n, G}]$$
%
%
and the \emph{twisted invariants}
\begin{align}\label{e:twistedinvariant} 
 \br{ \psi^{a_1} \phi^c_{ g_1}, \ldots, \psi^{a_n} \phi^c_{  g_n} }^{c,\bs}_{\h,n} 
 := \int_{W^c_{\h, n, G}(g_1\jj^{c},  \ldots, g_n\jj^{c}) } 
  \bs ( R \pi_* \oplus_{j=1}^N \cL_{j}) \prod_{i = 1}^n \psi_i^{a_i} .
\end{align}
We note that the shifting by $\jj^c$ is consistent with the definition of the untwisted
invariants in \eqref{e:1.0.3}.

There is an $\bs$-twisted pairing given by
\begin{equation} \label{e:1.1.4}
 \begin{split} 
 \langle \phi^c_{ g_1}, \phi^c_{ g_2} \rangle^{c, \bs}  := &\br{ \phi^c_{ g_1}, 
   \phi^c_{ g_2}, \phi^c_{ e}}^{c, \bs}_{0, 3} \\
  = &\exp \left(\sum_{j = 1}^N \chi \big(R \pi_*( \cL_{j}) \big)s^j_0 \right)
   \delta_{g_1 g_2 = \jj^{-2c}}/{\bar d}^N
 \end{split}
\end{equation} 
defined on $H^c[ \lambda, \lambda^{-1}]$.
The last equality follows easily from the definition and the fact that $\sMbar_{0,3}$ is a point.  This definition of the pairing is chosen to give a \emph{Frobenius algebra} structure on $H^c$.

\begin{lemma} \label{l:1.10} \begin{enumerate}
\item \label{i:sis0}
When $s^j_k = 0$ for all $j$ and $k$ we recover the untwisted $W^c$ invariants.  
In this case the pairing is simply 
$$\langle \phi^c_{ g_1\jj^{-c}}, \phi^c_{ g_2\jj^{-c}} \rangle^{c, 0} 
  = \frac{\delta_{g_1 = g_2^{-1}}}{{\bar d}^N}.$$
\item 
More generally,
\[
\langle \phi^c_{ g_1\jj^{-c}}, \phi^c_{ g_2\jj^{-c}} \rangle^{c, \bs} 
  =\exp \left( \sum_{j=1}^N \Big( \left\lfloor 1 - m_j(g_1) \right\rfloor
  +  \left\lfloor \tfrac{c c_j}{d} \right\rfloor \Big)s^j_0 \right)
  \frac{\delta_{g_1 = g_2^{-1}}}{{\bar d}^N}. 
\]
In particular, when $c c_j < d$, 
\begin{equation}\label{e:pairing} 
 \langle \phi^c_{ g_1\jj^{-c}}, \phi^c_{ g_2\jj^{-c}} \rangle^{c, \bs} 
  =\exp\Big(\sum_{j=1}^N \big( \left\lfloor 1 - m_j(g_1) \right\rfloor s^j_0 \big)\Big)\frac{\delta_{g_1 = g_2^{-1}}}{{\bar d}^N}.
\end{equation} 
The condition $c c_j < d$ holds in particular for the cases $c = 0$ or $1$.
\end{enumerate}
\end{lemma}

\begin{proof}
If $g_1 \neq g_2^{-1}$, the pairing is zero. 
\eqref{i:sis0} follows from \eqref{e:1.1.4}. 
(2) follows from a simple orbifold Riemann--Roch calculation. From equation \eqref{e:degree}, if $g_1 = g_2^{-1}$, we have   
\[
\begin{split}
 \deg ( |\cL_{j}|) &=\left\lfloor  \tfrac{c c_{j}}{d} \right\rfloor -  m_{j}(g_1) -  m_{j}(g_2) \\
 &= 
 \begin{cases}
   \big\lfloor  \frac{c c_{j}}{d} \big\rfloor & \text{if } m_{j}(g_1) =0 \quad (\text{and } m_{j}(g_2) =0), \\
   \big\lfloor  \frac{c c_{j}}{d} \big\rfloor -1 & \text{if } m_{j}(g_1) \neq 0 \quad (\text{and }  m_{j}(g_2) \neq 0).
 \end{cases}
\end{split}
\]
Then
\[
  \chi \big(R \pi_*( \cL_{j}) \big) = 
\begin{cases}
   \big\lfloor  \frac{c c_{j}}{d} \big\rfloor + 1 & \text{if } m_{j} (g_1) =0, \\
   \big\lfloor  \frac{c c_{j}}{d} \big\rfloor & \text{if } m_{j} (g_1) \neq 0.
\end{cases}
\]
\end{proof}

The data of the vector space $H^c [ \lambda, \lambda^{-1}]$ together with the 
$\bs$-twisted pairing are called the (equivariant) \emph{twisted state space}, denoted by $H^{c, \bs}$.
The twisted state space and the $\bs$-twisted $W^c$ invariants 
give an \emph{axiomatic Gromov--Witten theory}. See Section~\ref{ss:sympform} for the definition.

\section{Relations to other invariants}\label{s:rtoi}


Twisted invariants of $W^c$ structures give a general setting in which to describe other better known invariants.  
The first example is local invariants of a quotient of affine space and the second is the (genus zero) FJRW theory of Fermat-type Landau--Ginzburg pairs.

\subsection{Local invariants of $[\CC^N/G]$}

Given a pair $(Q, G)$ as before, we may define
\emph{local GW invariants} of $[\CC^N/G]$.  
For $g \in G$, let $\ii_g \in H^*_{CR}(BG)$ denote the fundamental class of the $g$-twisted sector of the inertia stack $I(BG)$.  
We view 
$$[\CC^N/G] \to BG$$ 
as a rank $N$ equivariant vector bundle, where $\CC^*$ acts on the $j$th factor with character $-\lambda_j$.
%

We define genus-$h$ local invariants of $[\CC^N/G]$ as
\begin{equation}  \label{e:local} 
\begin{split} 
 \br{ \psi^{a_1} \ii_{g_1}, \ldots, \psi^{a_n} \ii_{ g_n} }^{[\CC^N/G]}_{\h,n} 
  :=& \int_{[\sMbar_{\h,n}([\CC^N/G])]^{vir}} \prod_{i = 1}^n \psi_i^{a_i} \cup ev_i^*(\ii_{g_i})  \\
 :=&\int_{\sMbar_{\h,n}(BG)} \frac{\prod_{i = 1}^n \psi_i^{a_i} \cup ev_i^*(\ii_{g_i})}{e_{\CC^*}\left(R\pi_* f^* [\CC^N/G]\right)},  
\end{split}
\end{equation}
where $f$ is the universal map and $\pi$ the universal curve
\[
\begin{tikzcd}
  \cC \arrow{r}{f} \arrow{d}{\pi} &BG \\
  \sMbar_{\h,n}(BG) . 
\end{tikzcd}
\]
Note that the product $\prod_{i = 1}^n  ev_i^*(\ii_{g_i})$ simply specifies an open and closed substack of $\sMbar_{\h,n}(BG)$ over which to integrate.  
We would like to compare these integrals to $\bs$-twisted $W^c$ invariants.

First, observe that although $\sMbar_{\h,n}(BG)$ consists of \emph{representable} morphisms $\cC \to BG$, 
one may also consider the moduli space $\sMbar^{\bar{d}}_{\h,n}(BG)$ 
consisting of morphisms $\cC \to BG$ from a $\bar{d}$-stable curve which are not necessarily representable.  

\begin{lemma} \label{l:2.1}
There is a map 
$$\rho: \sMbar^{\bar{d}}_{\h,n}(BG) \to \sMbar_{\h,n}(BG)$$
where $\bar{d}$ is the period of $G$.
Furthermore, $\rho$ is an isomorphism over the open and dense locus consisting of non-nodal domain curves.
\end{lemma}

\begin{proof}
By the $r$-th root construction and in particular \cite[Theorem~4.1]{Cad}, 
there is a unique way of adding the $\mu_{\bar{d}}$ orbifold structure at the marked points.
The lemma follows.
\end{proof}

Thus we may instead define local invariants as integrals over $\sMbar^{\bar d}_{\h,n}(BG)$, 
where the integrand from \eqref{e:local} is pulled back via $\rho$.


\begin{lemma} \label{l:2.2}
There exists a natural morphism
\[\pi: W^0_{\h, n, G} \to  \sMbar^{\bar{d}}_{\h,n}(BG).\] This map
is a $\prod _{j=1}^N (\mu_{\bar d})$-gerbe.
\end{lemma}

\begin{proof}
We first show the existence of the morphism $\pi: W^0_{g, n, G} \to \sMbar^{\bar{d}}_{g,n}(BG)$.  
Given $(\cC, \cL_1, \ldots, \cL_N) \in W^0_{g, n, G}$, by construction there is a well defined $G$-action on each fiber of $\oplus_{j=1}^N \cL_j$, coming from the inclusion $G < \prod_j \mu_{\bar d}$.  
The fact that the associated principal bundle is a $G$-bundle follows from the definition of the $W^0$ structure in Definition~\ref{d:Wstr}.
This defines the morphism $\pi$.
The fact that $\pi$ is a gerbe can be seen from unraveling the definitions.
\end{proof}

\begin{remark} 
Alternatively, one can verify the degree of $\pi$ by the following observations.
Firstly, the degree of $W^0_{\h, n, G} \to  \sMbar_{\h, n}^{\bar d}$ is $|G|^{2h-1+n}/(\bar{d}^N)$.
In fact the fiber is $|G|^{2h - 1 + n}$ copies of $\prod_{j=1}^N B\mu_{\bar{d}}$, with the automorphisms coming from automorphisms of each line bundle $\cL_j$.
Secondly, the moduli space $\sMbar_{\h, n}^{\bar d}(BG)$ parameterizes curves $\cC$ in $\sMbar_{\h, n}^{\bar d}$ together with a homomorphism $\pi^{orb}_1( \cC) \to G$.  
Thus the fiber is given by $|G|^{2h - 1 + n}$ points, which parameterize maps $\pi^{orb}_1( \cC) \to G$.
Combining above degree counts, one gets the degree count for $\pi: W^0_{\h, n, G} \to  \sMbar^{\bar{d}}_{\h,n}(BG)$.
In fact, a detailed analysis of the above two steps gives another verification of the second statement of Lemma~\ref{l:2.2}.
\end{remark}

Thus by the projection formula, integrals over $\sMbar^d_{\h,n}(BG)$ coincide with those over $W^0_{\h, n, G}$ up to a factor of ${\bar d}^N$.  
If we consider $\bs$-twisted invariants with 
\begin{equation*}
e^{s^j_0} = -\frac{1}{\lambda_j}\text{ and } s^j_k = (k-1)!/\lambda^k_j \text{ for }1 \leq j \leq N, \: k > 0,
\end{equation*} 
then $\bs([\CC^N/G]) = 1/e_{\CC^*}([\CC^N/G])$.  
We finally arrive at the following relation.

\begin{corollary} \label{c:2.4}
\begin{equation*} 
  \bar{d}^{N} \br{ \psi^{a_1} \phi^0_{ g_1}, \ldots, \psi^{a_n} \phi^0_{  g_n} }^{0,\bs}_{\h,n} = \br{ \psi^{a_1} \ii_{g_1}, \ldots, \psi^{a_n} \ii_{ g_n} }^{[\CC^N/G]}_{\h,n}.
\end{equation*}
\end{corollary}

In this way $\bs$-twisted invariants of $W^0$ structures specialize to the local Gromov--Witten theory of a point.  

\subsection{FJRW invariants of Fermat polynomials}
Given a Landau--Ginzburg pair $(Q,G)$, Fan, Jarvis and Ruan have constructed a cohomological field theory called FJRW theory. 
The corresponding numerical invariants are likewise called FJRW invariants. 
\begin{definition}  \label{d:2.5}
Given a Landau--Ginzburg pair $(Q,G)$, the \emph{narrow FJRW state space} is given by 
\[
 \cH_{FJRW}(Q, G) := \oplus_{g \in \hat G} \CC \varphi_g,
\] 
where 
$$\hat G : =\{g \in G| g\jj \text{ fixes only the origin in } \CC^N\}$$
and $\varphi_g$ is a vector formally associated to $g \in \hat{G}$.
\end{definition}
\begin{remark}
There is a larger FJRW state space which includes the so-called \emph{broad sectors}, 
(subspaces corresponding to those $g \notin \hat G$) 
but we will restrict ourselves here to the narrow state space without loss of information.  
In fact all invariants involving broad sectors vanish due to the so called \emph{Ramond vanishing} property.  
See Remark~2.3.2 of \cite{ChR}.
\end{remark}

Similar to the case of local GW theory, one may specialize $\bs$-twisted invariants of $W^1$ structures to recover genus zero FJRW invariants.  
When $c = 1$, there is 
a birational map from $W^1_{\h, n, G}$ to the FJRW moduli space, denoted $\sW_{\h, n, G}$.  
Again we may define FJRW invariants as integrals over  $W^1_{\h, n, G}$ by pulling back classes on $\sW_{\h, n, G}$ via this map. 

The construction of the FJRW virtual cycle is in general quite complicated, but in case $Q$ is a Fermat polynomial and the genus is zero the situation simplifies greatly.  
In this case one can prove \cite{FJR1} that 
$$R^0\pi_* (\oplus_{j=1}^N \cL_j) = 0$$ and 
$$-R\pi_* (\oplus_{j=1}^N \cL_j) =  R^1(\oplus_{j=1}^N \cL_j)[-1] $$ 
is a vector bundle.  Then by axiom (5a) of \cite[Theorem~4.1.8]{FJR1}, 

\begin{equation} \label{e:2.2.1}
\begin{split}
\br{ \psi^{a_1} \varphi_{0, g_1}, \ldots, \psi^{a_n} \varphi_{0,  g_n} }^{(Q,G)}_{0,n} 
&:=  {\bar d}^N \int_{W_{0, n, G}(g_1\jj, \ldots, g_n\jj)}
\frac{\prod_{i = 1}^n \psi_i^{a_i}}{e \left(R\pi_* ( \oplus_{{j'}=1}^N   \cL_{j'})^{\vee} \right)} \\
 \quad& =(-1)^{\chi ( \oplus \cL_{j'})} {\bar d}^N \int_{W_{0, n, G}(g_1 \jj, \ldots, g_n \jj)}
\frac{\prod_{i = 1}^n \psi_i^{a_i}}{e \left(R\pi_* ( \oplus_{{j'}=1}^N   \cL_{j'})\right)}
\end{split}
\end{equation}
for $g_i \in \hat G$.

Similar to the case of local GW theory, consider $\bs '$-twisted invariants with 
\begin{equation}\label{e:bs2} e^{{s^j_0} '} = \frac{1}{\lambda_j}\text{ and } {s^j_k} ' = (k-1)!/\lambda^k_j \text{ for } 1 \leq j \leq N, \: k > 0.\end{equation}
Note that $e^{{s^j_0} '}$ differs from $e^{s^j_0}$ by a sign, this will alter the overall sign of our invariants by $(-1)^{\chi ( \oplus \cL_j)}$.  
We obtain a relation between the nonequivariant limit of $\bs '$-twisted invariants and FJRW invariants:

\begin{corollary} \label{c:2.6} If $g_i \in \hat G$ for all $i$,
\begin{equation*} 
\br{ \psi^{a_1} \varphi_{g_1}, \ldots, \psi^{a_n} \varphi_{ g_n} }^{(Q, G)}_{0,n} =\lim_{\lambda \mapsto 0} 
{\bar d}^N \br{ \psi^{a_1} \phi^0_{ g_1}, \ldots, \psi^{a_n} \phi^0_{  g_n} }^{1,\bs '}_{0,n} 
\end{equation*}

\end{corollary}

The inner product on the narrow state space (Definition~\ref{d:2.5}) is defined as in \eqref{e:1.1.4}.
Due to narrowness condition, the pairing will not degenerate at the non-equivariant limit.

\section{Givental's symplectic formalism}\label{ss:sympform}

%
%
Motivated by the common structures in Gromov--Witten theory, Givental \cite{G3} has developed a formalism for dealing with ``Gromov--Witten-like'' theories, which we shall refer to as axiomatic Gromov--Witten theories (Definition~\ref{d: axiomatic}).  Although we will not give a complete description of such a theory here, we collect below several of the important facts which shall be used in the what follows.  We refer the interested reader to 
\cite{G3} for more information.

Let $\square$ denote the data of a state space $\left(H^\square, \langle - ,- \rangle_\square\right)$ and invariants 
\[
\langle \psi^{a_1} \beta_{i_1}, \ldots,  \psi^{a_n} \beta_{i_n} \rangle^\square_{g, n}\] 
for $\set{\beta_i}_{i\in I}$ a basis of $H^\square$. The examples of $\square$ to have in mind are Gromov--Witten theory, FJRW theory, or that of $\bs$-twisted $W^c$ structures.
  
We may define formal generating functions of $\square$ invariants. Let $\bt = \sum_{i \in I} t^i \beta_i$ represent a point of $H^\square$ written in terms of the basis.  
For notational convenience denote the formal series $\sum_{k \geq 0} \bt_k \psi^k $ as $\bt(\psi)$.  
Define the genus $g$ generating function by
\[
\cF_g^\square := \sum_n \frac{1}{n !} \langle \bt(\psi), \ldots,  \bt(\psi) \rangle^\square_{g, n}.
\] 
Let $\cD$ denote the \emph{total genus descendent potential},
\[ 
\cD^\square := \exp \left(\sum_{g \geq 0} \hbar^{g-1} \cF_g^\square\right).
\]
GW theory, FJRW theory, and $\bs$-twisted $W^c$ invariants all share a similar structure.  In particular, their genus-$g$ generating functions satisfy three differential equations, the so--called \emph{string equation} (SE), \emph{dilation equation} (DE), and \emph{topological recursion relation} (TRR).  (See \cite{YPnotes} for an explicit description of each.)

\begin{definition}\label{d: axiomatic}
We call $\square$ an \emph{axiomatic GW theory} if the correlators satisfy the SE, DE, and TRR.\end{definition}
\begin{remark}
For the proof that  Gromov--Witten theory satisfies the above equations see \cite{Ts}, in the case of FJRW theory see \cite{FJR1}.  That $\bs$-twisted $W^c$ structure invariants give an axiomatic GW theory follows from Theorem~\ref{t:symplectictransformation} and the corresponding statement for untwisted invariants.
\end{remark}

We can use this extra structure to rephrase the genus zero data in terms of Givental's \emph{overruled Lagrangian cone}.  For a more detailed exposition of what follows we refer the reader to Givental's original paper on the subject (\cite{G1}).

Let $\sV^\square$ denote the vector space $H^\square((z^{-1}))$, equipped with the symplectic pairing 
\begin{equation*}
\Omega_\square(f_1, f_2) := \Res_{z=0}\langle f_1(-z), f_2(z)\rangle_\square.
\end{equation*}
$\sV^\square$ admits a natural polarization $\sV^\square = \sV^\square_+ \oplus \sV^\square_-$ defined in terms of powers of $z$: 
\begin{align*}\sV^\square_+ &= H^\square[z], \\ \sV^\square_- &= z^{-1}H^\square[[z^{-1}]].\end{align*} 
We obtain Darboux coordinates  $\set{q_k^i, p_{k,i}}$ with respect to the polarization on $\sV^\square$ by representing each 
element of $\sV^\square$ in the form 
\[
\sum_{k \geq 0}\sum_{i \in I} q_k^i \beta_i z^k + \sum_{k \geq 0}\sum_{i \in I} p_{k,i}\beta^i (-z)^{-k-1}
\]
One can view $\cF^\square_0$ as the generating function of a Lagrangian subspace $\sL^\square$ of $\sV^\square$.  Let $\beta_0$ denote the unit in $H^\square$, and make the change of variables (the so--called Dilaton shift)
\[
q_1^0=t_1^0-1 \quad q_k^i=t_k^i \text{ for }(k,i)\neq (1,0).
\]
Then the set 
\begin{equation}\label{e:lagcone}
\sL^\square :=\set{\bp =d_\bq \cF^\square_0}
\end{equation}
defines a Lagrangian subspace.  More explicitly, $\sL^\square$ contains the points of the form
\begin{equation}\label{e:conedescr}
-\beta_0 z+ \sum_{\substack{k \geq 0 \\ i \in I}} t_k^i \beta_i z^k +\sum_{\substack{a_1, \ldots , a_n,  a\geq 0 \\ i_1, \ldots , i_n, i \in I}} \frac{t^{i_1}_{a_1}\cdots t^{i_n}_{a_n}}{n!(-z)^{a+1}}\langle \psi^a\beta_i ,\psi^{a_1}\beta_{i_1},\dots,\psi^{a_n}\beta_{i_n}\rangle_{0, n+1}^\square\beta^i.
\end{equation}
Because $\cF^\square_0$ satisfies the
SE, DE, and TRR, $\sL^\square$ will take a special form.  In fact, $\sL^\square$ is a cone satisfying the condition that for all $f \in \sV^\square$,
\begin{equation}\label{e:overruled}\sL^\square \cap L_f = zL_f\end{equation} 
where $L_f$ is the tangent space to $\sL^\square$ at $f$. Equation~\eqref{e:overruled} justifies the term overruled, as each tangent space $L_f$ is  filtered by powers of $z$:
\[
L_f \supset zL_f \supset z^2L_f \supset \cdots
\] 
and $\sL^\square$ itself is ruled by the various $zL_f$.
The codimension of $zL_f$ in $L_f$ is equal to $\dim(H^\square)$.  

A generic slice of $\sL^\square$ parameterized by $H^\square$, i.e. 
\[
\{f(\bt)| \bt \in H^\square\} \subset \sL^\square,
\]
will be transverse to the ruling. Given such a slice, we can reconstruct $\sL^\square$ as
\begin{equation}\label{e:generic}
\sL^\square = \set{zL_{f(\bt)}| \bt \in H^\square}.
\end{equation}

Givental's $J$--function is defined in terms of the intersection
\[
\sL^\square \cap -\beta_0z \oplus H \oplus \sV^{-}.
\]
More explicitly, the $J$--function is given by
\[
J^\square(\bt, -z) = -\beta_0z + \bt + \sum_{n \geq 0} \sum_{ i \in I} \frac{1}{n!} \br{ \frac{\beta_i}{-z - \psi}, \bt, \ldots, \bt }^\square_{0, n+1} \beta^i.
\] 
In other words, we obtain the $J$--function by setting $t_k^i=0$ in \eqref{e:conedescr} whenever $k>0$. 

In \cite{G3} it is shown that the image of $J^\square(\bt, -z)$ is transverse to the ruling of $\sL^\square$, so $J^\square(\bt, -z)$ is a function satisfying \eqref{e:generic}.  Thus the ruling at $J^\square(\bt, -z)$ is spanned by the derivatives of $J^\square$, i.e.

\begin{equation}\label{e:Jgens}zL_{J^\square(\bt, -z)} = \set{J^\square(\bt, -z) + z\sum c_i(z) \frac{\partial}{\partial t^i} J^\square(\bt, -z) | c_i(z) \in \CC[z]}.
\end{equation}

By the string equation, $ z\frac{\partial}{\partial t^0} J^\square(\bt, z) = J^\square(\bt, z)$, so \eqref{e:Jgens} simplifies to
\begin{equation}\label{e:Jgens2}
zL_{J^\square(\bt, -z)} = \set{ z\sum c_i(z) \frac{\partial}{\partial t^i} J^\square(\bt, -z) | c_i(z) \in \CC[z]}.
\end{equation}


\section{Twisted theory from untwisted theory} \label{s:correspondences}

Here a 
correspondence between $\bo$-twisted $W^c$ 
invariants and $\bs$-twisted $W^c$ invariants is presented
using the language of
Givental's symplectic formalism.

\subsection{Grothendieck--Riemann--Roch for $r$-spin curves}

We recall A.~Chiodo's
Grothendieck--Riemann--Roch calculation for $r$-spin curves \cite{Ch}
which will then be adapted to the setting of $W^c$ structures.
First, we set notation.

Let 
$$\mm_1, \ldots, \mm_n \in \{ \frac{0}{r}, \ldots, \frac{r-1}{r} \}$$ 
be multiples of $1/r$ with $0 \leq \mm_i < 1$.  
Let 
$$\sAbar_{\h, n}^{(r,c)}(\mm_1, \ldots, \mm_n)$$ 
denote the component of $\sAbar_{\h, n}^{(r,c)}$ such that the multiplicity 
of the isotropy at $p_i$ is $\mm_i$.  
By \eqref{e:numcon}, if $(c/r)(2h-2 + n) - \sum_{i = 1}^n \mm_i \in \ZZ$ 
then $\sAbar_{\h, n}^{(r,c)}(\mm_1, \ldots, \mm_n)$ will be nonempty.  
Let $\Sing$ denote the stack classifying nodal curves equipped with an $r$th root, along with a choice of node.  
By specifying a branch at the node, we obtain a double cover $\Sing ' \to \Sing$.  
The stack $\Sing$ maps to $\sAbar_{\h, n}^{(r,c)}(\mm_1, \ldots, \mm_n)$, 
composing we obtain
\[ \iota: \Sing ' \to \sAbar_{\h, n}^{(r,c)}(\mm_1, \ldots, \mm_n).\]
The stack $\Sing '$ decomposes as a disjoint union of substacks 
\[
 \Sing := \sqcup_{0 \leq q < (r-1)/r} \Sing'_q ,
\]
determined by the multiplicity at the node.  
Namely, given a point $p \in \Sing '$, let $\cL \to \cC$ denote the corresponding $r$th root.  
The isotropy at the distinguished node acts on the restriction of $\cL$ to the first branch.  
Let $q(p)$ denote the multiplicity of this action.  
This multiplicity is constant on connected components, so we define $\Sing_q '$ to be the subset of $\Sing'$ where the multiplicity is $q$.  
We further denote by 
\[
  \iota_q: \Sing_q ' \to \sAbar_{\h, n}^{(r,c)}(\mm_1, \ldots, \mm_n)
\]
 the restriction of the map $\iota$.
There are line bundles over $\Sing '$ whose fibers are the cotangent space of first branch of the coarse curve at the node 
and the cotangent space of the second branch of the coarse curve at the node.  
Let $\psi, \hat \psi \in H^2(\Sing ', \QQ)$ denote their respective first Chern classes.  
Finally, define the class
\[ \gamma_r = \sum_{i + j = r} (-\psi)^i {\hat \psi}^j.\]  

Let $B_k(x)$ denote the $k$th \emph{Bernoulli polynomial} defined by 
$$\sum_{k\geq 0} B_k(x) z^k/k! = ze^{zx}/(e^z - 1).$$
Chiodo proves the following generalization of Mumford's Grothendieck--Riemann--Roch calculation.

\begin{theorem}[\cite{Ch}]\label{t:ch}  
Let $\cL$ denote the universal line bundle over the universal curve $\pi: \cC \to \sAbar_{\h, n}^{(r,c)}(\mm_1, \ldots, \mm_n)$.  Then
\begin{equation*}
 \begin{split}
 &\ch(R \pi_* \cL) = \sum_{k \geq 0} \\
 &\quad \left( \frac{B_{k+1}(c/r)}{(k+1)!}\kappa_k - \sum_{i = 1}^n \frac{B_{k+1}(\mm_i)}{(k+1)!} \psi_i^k + \frac{1}{2} \sum_{q = 0}^{r-1} \frac{r B_{k+1}(q)}{(k+1)!}(\iota_q)_* (\gamma_{k-1})\right),
 \end{split}
\end{equation*}
where $\kappa_k$ are the $\kappa$ classes (cf.~\eqref{e:kappa}).
\end{theorem}

\subsection{Twisted from untwisted invariants} \label{s:4.3}

An important application of Givental's symplectic formalism is that 
it enables one to systematically relate twisted and untwisted invariants.  
This has been used to great effect in Gromov--Witten theory by Coates and Givental \cite{CG}.
We will adapt that method to the setting of $W^c$ structures.

In the spirit of \cite{CG} and \cite{CZ}, the above theorem gives an
explicit relationship between twisted and untwisted $W^c$ invariants, 
which may be expressed most neatly in the language of Givental's symplectic formalism.  
\emph{We will assume without further comment in what follows that 
we are in the situation $c c_j < d$ for all $j$.}

\begin{notation}
For the purposes of more clearly stating the main theorem, we let $m_j(\phi^c_{ g})$ denote the multiplicity corresponding to 
\emph{an insertion of $\phi^c_{ g}$}.  Due to the shifting in \eqref{e:twistedinvariant},  
an $\bs$-twisted invariant with a $\phi^c_{ g}$ insertion corresponds to a marked point where the action of the isotropy on the fiber is by $g\jj^c$.  
Thus $m_j(\phi^c_{ g}) = m_j(g\jj^c)$, the multiplicity of $g\jj^c$ on the $j$th line bundle.
\end{notation}

\begin{theorem}\label{t:symplectictransformation}  
For any Landau--Ginzburg pair $(Q, G)$,
let $\Delta^c: \sV^{c, \bo} \to \sV^{c, \bs}$ 
be the symplectic transformation defined by
\[ 
 \Delta^c:= \bigoplus_{g \in G} \prod_{j = 1}^N \exp 
 \left( \sum_{k \geq 0} s^j_k \frac{B_{k+1}(m_j(\phi^c_{g}))}{(k+1)!}z^k\right)
\] 
and let $\widehat \Delta^c$ denote the quantization of $\Delta^c$, 
as defined in \cite{CG} (or \cite{YPnotes}).
Then, 
\begin{enumerate}
\item$\widehat \Delta^c$ relates the twisted and untwisted total 
descendent potentials (of all genera)
\begin{equation}\label{e:quantrel}
\cD^{c, \bs} = \widehat \Delta^c \cD^{c, \bo}.
\end{equation}
\item
$\Delta^c$ relates the twisted and untwisted Lagrangian cones
\[
\sL^{c, \bs} = \Delta^c \sL^{c, \bo}.
\]  
\end{enumerate}
\end{theorem}

\begin{remark}
In the above notation, the direct sum means simply that we act on the $\CC((z))$-span of $\phi^c_{ g}$ by multiplication by 
$$\prod_{j = 1}^N \exp \left( \sum_{k \geq 0} s^j_k \frac{B_{k+1}(m_j(\phi^c_{g}))}{(k+1)!}z^k\right).$$
\end{remark}

\begin{remark}
In the $r$-spin ($Q = x^r$, $G = \langle \jj \rangle$, $c = 1$) case, restricting to narrow sectors, the above result was proven in \cite{CZ}.
%
%
A similar generalization to the above was given in the case of the Fermat quintic in \cite{ChR}, 
with a slight difference due to their definition of the $W^c$ moduli space at broad sectors.
\end{remark}

\begin{proof}  
The proof follows the method first used in \cite{CG}, and is a straightforward generalization of \cite{CZ}.  
We first remark that (1) implies (2), as (2) is nothing but a semi-classical limit of (1)
(see \cite{CG} and \cite{CPS}).
Therefore, it is enough to show (i).

 Viewing both sides of \eqref{e:quantrel} as functions with respect to the formal parameters $s^j_k$ for $k \geq 0$, 
 it suffices to show that they satisfy the same differential equation with respect to $s^j_k$.  
 Note first that both sides of \eqref{e:quantrel} satisfy the same initial condition, i.e. when $\bs = 0$ the two are equal.  We next claim that both sides satisfy
 \begin{equation} \label{e:4.3.2}
 \frac{\partial \Phi}{\partial s^j_k} = P_k^{(j)} \Phi,
 \end{equation}
 where
 \begin{align*} P_k^{(j)} = 
 \frac{B_{k+1}(c/\bar{d})}{(k+1)!}\frac{\partial}{\partial t^{e}_{k+1}} &- \sum_{\substack{ a\geq 0 \\ g \in G}} \frac{B_{k+1}(m_j(\phi^c_{g}))}{(k+1)!} t^g_a \frac{\partial}{\partial t^g_{a + k}} \\ &+ \frac{\hbar}{2} \sum_{\substack{a+ a' = k-1\\ g, g ' \in G}} (-1)^{a '} \eta^{g, g'} \frac{ B_{k+1}(m_j(\phi^c_{g}))}{(k+1)!} \frac{\partial^2}{\partial t^g_a \partial t^{g'}_{a'}}.
 \end{align*}  
 Here $\eta_{g, g'} = \langle \phi^c_{ g}, \phi^c_{ g'} \rangle^{c, \bs}$ is the pairing matrix, and upper indices denote the corresponding coordinate of the dual matrix.  
 
 That the right side of \eqref{e:quantrel} satisfies this equation is a direct consequence of the definition of $\widehat \Delta^c$ (\cite{CZ}). 
The proof of Theorem~\ref{t:symplectictransformation} will be complete
after we show that the left side 
of \eqref{e:quantrel} also satisfies \eqref{e:4.3.2}.

%
%

By Theorem~\ref{t:ch}, differentiating $\cF^{c, \bs}$ with respect to $s_k^j$ has the effect of adding a factor of 
\begin{equation}\label{e:integrand}
\begin{split}
  \bigg( \frac{B_{k+1}(c/\bar{d})}{(k+1)!}\kappa_k &- \sum_{i = 1}^n \frac{B_{k+1}(m_j(\phi^c_{g_i}))}{(k+1)!} \psi_i^k \\
  &+ \frac{1}{2} \sum_{\g \in G} \frac{{\bar{d}}^N B_{k+1}(m_j(\phi^c_{\g}))}{(k+1)!}(\iota_\g)_* (\gamma_{k-1}) \bigg)
\end{split}
\end{equation}
to each integrand in the generating function $\cF^{c, \bs}$.

We will investigate the contribution of each term in the above expression to $\frac{\partial }{\partial s^j_k} \cF^{c, \bs}$.

Step 1: Recall the class $\kappa_k$ is defined as the pushforward of 
$\psi_{n+1}^{k+1}$ under the map 
\begin{equation} \label{e:kappa}
 \coprod_{g_1, \ldots , g_n \in G} W^c_{\h, n+1}(g_1, \ldots, g_n, \jj^c) \to \coprod_{g_1, \ldots , g_n \in G} W^c_{\h, n}(g_1, \ldots, g_n).
\end{equation}  
By the projection formula
\begin{align*}
&\int_{W^c_{\h, n, G}(g_1\jj^{c},  \ldots, g_n\jj^{c}) } \kappa_k \cup \bs ( R \pi_* \oplus_{j=1}^N \cL_j) \prod_{i = 1}^n \psi_i^{a_i} \\
&= \int_{W^c_{\h, n+1, G}(g_1\jj^{c},  \ldots, g_n\jj^{c}, \jj^c) } \psi_{n+1}^{k+1} \cup \bs ( R \pi_* \oplus_{j=1}^N \cL_j) \prod_{i = 1}^n \psi_i^{a_i} \\ &= \langle \psi^{a_1} \phi^c_{ g_1}, \ldots , \psi^{a_n} \phi^c_{ g_n}, \psi^{k+1} \phi^c_{e}\rangle_{\h, n}^{c, \bs}
\end{align*}
Thus the first term in \eqref{e:integrand} contributes a summand \[ \frac{B_{k+1}(c/\bar{d})}{(k+1)!}\frac{\partial}{\partial t^{e}_{k+1}} \cF^{c, \bs}\] to $\frac{\partial }{\partial s^j_k}\cF^{c, \bs}$.

Step 2:  It can be seen immediately that adding a factor of 
$$\sum_{i = 1}^n \frac{B_{k+1}(m_j(\phi^c_{g_i}))}{(k+1)!} \psi_i^k$$ 
to the integrand of
\[
 \int_{W^c_{\h, n, G}(g_1\jj^{c},  \ldots, g_n\jj^{c}) } 
  \bs ( R \pi_* \oplus_{j=1}^N \cL_j) \prod_{i = 1}^n \psi_i^{a_i} 
\]
 for every such integral in $\cF^{c, \bs}$ is equivalent to acting on 
$\cF^{c, \bs}$ by 
$$ \sum_{a, g} \frac{B_{k+1}(m_j(\phi^c_{g}))}{(k+1)!} 
  t^g_a \frac{\partial}{\partial t^g_{a + k}}.$$  
Thus the second term in \eqref{e:integrand} contributes a summand 
\[
  - \sum_{\substack{ a\geq 0 \\ g \in G}} \frac{B_{k+1}(m_j(\phi^c_{g}))}{(k+1)!} t^g_a \frac{\partial}{\partial t^g_{a + k}} \cF^{c, \bs}
\] 
to $\frac{\partial }{\partial s^j_k}\cF^{c, \bs}$.

Step 3:  
Let $\Sing '$ denote the double cover of the nodal locus 
$\Sing \subset W_{\h, n, G}$ in analogy to the $\bar{d}$-spin case.  
In this case $\Sing ' $ splits as a disjoint union $\coprod_{\g \in G} \Sing_\g '$ where the action of the isotropy group at the node of the distinguished component on the fiber of $\oplus_{j = 1}^N \cL_j$ is by $\g\jj^c$.  Let 
$$\iota_\g: \Sing' \to W_{\h, n, G}$$ 
denote the restriction of $\iota$ to $\Sing_\g '$.
 Let 
$$\cF^{c, \bs} := \sum_{\h \geq 0} \hbar^{\h-1} \cF^{c, \bs}_\h.$$   

The stack $\Sing '$ splits into two open and closed subsets based on whether the node is separating or non-separating.  Let 
$$\iota_{\g, irr}: \Sing_{\g, irr} ' \to W^c_{\h, n, G}(g_1\jj^c, \ldots, g_n\jj^c)$$ 
denote the restriction of $\iota_{\g}$ to the non-separating locus. It is
easy to see that $\iota_{g, irr}$ factors through 
$W^c_{\h-1, n+2, G}(g_1\jj^c, \ldots, g_n\jj^c, \g\jj^c, \g^{-1}\jj^{-c})$:
\[
 \begin{tikzcd}
  \, &W^c_{\h-1, n+2, G}(g_1\jj^c, \ldots, g_n\jj^c, \g\jj^c, \g^{-1}\jj^{-c})
  \arrow{d}{\rho} \\
 \Sing'_{\g, irr} \arrow{r}{\iota_{\g, irr}} \arrow{ru}{\mu_{\g, irr}}
 &W^c_{\h, n, G}(g_1\jj^c, \ldots, g_n\jj^c)
 \end{tikzcd}
\]
induced by normalizing the universal curve. Let 
$$\oplus_{j = 1}^N \cL_j \to \cC 
  \stackrel{\pi}{\to} W^c_{\h, n, G}(g_1\jj^c, \ldots, g_n\jj^c)$$
 and 
$$\oplus_{j = 1}^N \bar \cL_j \to \bar \cC 
  \stackrel{\bar{\pi}}{\to} W^c_{\h-1, n+2, G}(g_1\jj^c, \ldots, g_n\jj^c, \g\jj^c, \g^{-1}\jj^{-c})$$
be the universal bundles over the universal curves, then 
there is a natural morphism 
$$\nu: \bar \cC \to \cC$$ 
via the normalization of the curves.
The key point is that because the pullback
$\nu^* \omega_{\cC, \log}$ is equal to $\omega_{\bar \cC, \log}$, the line bundle
$\cL_j$ will pull back to $\bar \cL_j$. If 
$$n: \Sing_{\g, irr} ' \to \cC$$ is the morphism induced by the nodal locus, 
the normalization exact sequence yields 
$$0 \to R \pi_* \cL_j \to R \bar{\pi}_* \nu^* \cL_j \to  n^* \cL_j \to 0.$$  
%

If $\g\jj^c$ acts nontrivially on the $j$th line bundle then $\ch(n^* \cL_j) = 0$.  Otherwise,
$n^* \cL_j$ is a root of (a power of) $n^* \omega_{\cC, \log}$ which is trivial via the residue map.  
Thus $n^* \cL_j$ is rationally trivial and $\ch( n^* \cL_j )= 1$.
We arrive at the formula
\[
 {\mu_{\g, irr}}_* \iota_{\g, irr}^*(\ch_k(R \pi_* \cL_j)) = 
 \begin{cases}
  \ch_k(R \bar{\pi}_* \bar \cL_j) & \text{for } k > 0 \text{ or } 
  m_j(\g \jj^c) \neq 0 \\ 
  \ch_0(R \bar{\pi}_* \bar \cL_j) - 1 & \text{otherwise}
 \end{cases},
\]
which yields the simple relation (cf.\ Equation~\eqref{e:pairing})
\begin{align*}
 &{\mu_{\g, irr}}_* \iota_{\g, irr}^* ( \bs( R \pi_* \oplus_{j = 1}^N \cL_j)) \\
 = &\exp(- \sum_{j=1}^N \lfloor 1 - m_j(\phi_g^c) \rfloor 
 s^j_0 ) \bs( R \bar{\pi}_* \oplus_{j = 1}^N \bar \cL_j)\\
= &\frac{\eta^{\g, \g^{-1}}}{{\bar d}^N}\bs( R \bar{\pi}_* \oplus_{j = 1}^N \bar \cL_j).
\end{align*}  
Thus integrals involving a pushforward via $\iota_{\g, irr}$ may instead be calculated as integrals over 
$W^c_{\h-1, n+2, G}(g_1\jj^c, \ldots, g_n\jj^c, \g\jj^c, \g^{-1}\jj^{-c})$.

This implies that
\begin{align*}
&\int_{W^c_{\h, n, G}(g_1\jj^{c},  \ldots, g_n\jj^{c}) } {\bar d}^N {\iota_{\g, irr}}_*( \gamma_{k-1})\cup \bs ( R \pi_* \oplus_{j=1}^N \cL_j) \prod_{i = 1}^n \psi_i^{a_i} \\
=& \eta^{\g, \g^{-1}} \int_{W^c_{\h-1, n, G}(g_1\jj^{c},  \ldots, g_n\jj^{c}, \g \jj^c, \g^{-1}\jj^{-c}) }  {\mu_{\g, irr}}_*\gamma_{k-1}\cup \bs ( R \bar{\pi}_* \oplus_{j=1}^N \cL_j) \prod_{i = 1}^n \psi_i^{a_i}\\
=&  \sum_{a+ a' = k-1} (-1)^{a '} \eta^{\g, \g^{-1}} \langle \psi^{a_1} \phi^c_{ g_1}, \ldots, \psi^{a_n} \phi^c_{ g_n}, \psi^{a} \phi^c_{\g}, \psi^{a'} \phi^c_{\g^{-1}} \rangle_{\h-1, n+2}^{c, \bs}.
\end{align*}
Therefore, the \emph{non-separating} part of the third term in \eqref{e:integrand} contributes a summand 
\[ \frac{\hbar}{2} \sum_{\substack{a + a' = k-1\\ g, g' \in G}} (-1)^{a '} \eta^{g, g'} \frac{ B_{k+1}(m_j(\phi^c_{g}))}{(k+1)!} \frac{\partial^2}{\partial t^g_a \partial t^{g'}_{a'}}\cF^{c, \bs}\] 
to $\frac{\partial}{\partial s^j_k} \cF^{c, \bs}$.

A similar argument shows that the \emph{separating} part of the third term in \eqref{e:integrand} contributes a summand 
\[ \frac{\hbar}{2} \sum_{\substack{a + a' = k-1\\ g, g' \in G}} (-1)^{a '} \eta^{g, g'} \frac{ B_{k+1}(m_j(\phi^c_{g}))}{(k+1)!} \frac{\partial}{\partial t^g_a }\cF^{c, \bs}\frac{\partial}{\partial t^{g'}_{a'}}\cF^{c, \bs}\] 
to $\frac{\partial}{\partial s^j_k} \cF^{c, \bs}$.  

Finally, adding all these contributions, we conclude that 
\[\frac{\partial}{\partial s^j_k} \cD^{c, \bs} =
 P_k^{(j)}\cD^{c, \bs}\]
as desired.

\end{proof}

\section{Correspondences}\label{s:cc2}

\subsection{Identification of state spaces}
Let 
$$\sL^{c, \bo} \subset \sV^{c, \bo} \quad \text{and} \quad  \sL^{c, \bs} \subset \sV^{c, \bs}$$
denote the Lagrangian cones (see \eqref{e:lagcone}) of untwisted and 
$\bs$-twisted $W^c$ invariants respectively.  
%
%
%

We have the following identification of twisted and untwisted state spaces as vector spaces with inner products.

\begin{lemma}  
If $c c_j < d$ for $1 \leq j \leq N$, 
the map, 
\begin{align*}
  i_c: H^{0, \bs} &\to H^{c, \bs} \\ 
  \phi^0_{ g} &\mapsto \phi^c_{ g \jj^{- c}}
\end{align*} 
is an isomorphism of inner product spaces.
\end{lemma}

\begin{proof}
This follows immediately from Lemma~\ref{l:1.10}.
\end{proof}

Note that in this case $i_c$ extends to an isomorphism of 
the symplectic spaces $\sV^{0, \bs} \cong \sV^{c, \bs}$.
We will use the notation
\[
J^{c, \bs} (\bt, z) = \beta_0 z + \bt + \sum_{n \geq 0} \sum_{ i \in I} \frac{1}{n!} \br{ \frac{\beta_i}{z - \psi}, \bt, \ldots, \bt }^{c, \bs}_{0, n+1} \beta^i.
\] 
for the $J$-function of the twisted theories.

\subsection{Untwisted invariants}
We will first consider the case $s_k = 0$ for all $k$ (Lemma~\ref{l:1.10}).  
By Equation~\eqref{e:Jgens}, the genus zero part of the theory (i.e. the Lagrangian cone) is determined by the $J$-function.  
Let 
$\bt_c := \sum_{g \in G} t^{g}_{c} \phi^c_{ g}$ denote a point in the state space.
The $J$-function $J^{c, \bo}(\bt, z)$ may be directly calculated. 

\begin{lemma} \label{l:jc0}
\begin{align*}
 J^{c, \bo}(\bt^{\bo}_c, z) = \sum_{\{ a_g \geq 0\}_{g \in G}} 
   z^{ 1- \sum a_g}\prod_{g \in G}\frac{ (t^g_{c})^{a_g}}{a_g!} \phi^c_{ \prod g^{a_g}}. 
\end{align*}
\end{lemma}

\begin{proof} 
By pushing forward via the forgetful morphism, the $\bo$-twisted invariants may be calculated over $\sMbar_{0, n}$.  
The moduli space $W^c_{0, n, G}(g_1 \jj^{c},  \ldots, g_n\jj^{c})$ is nonempty if 
$\prod_{i = 1}^n g_i = \jj^{-2c}$ by \eqref{e:numcon}.  In this case the degree of the map to $\sMbar_{0, n}$ is $1/{\bar d}^N$.
Thus by \eqref{e:1.0.3}, and the standard formula for integrals of $\psi$-classes over $\sMbar_{0, n}$, the invariant $\br{ \psi^{a} \phi^c_{ g_0}, \phi^c_{g_1},  \ldots,  \phi^c_{  g_n} }^c_{0,n+1}$ is equal to $1 / \bar{d}^N$  exactly when $g_0 =  \jj^{-2c} \prod_{i = 1}^n g_i^{-1}$ and $a = n - 2$ and is zero otherwise.  By Lemma~\ref{l:1.10}, the dual of $\phi_{g_0}$ is $\bar{d}^N \phi_{g_0^{-1}\jj^{-2c}}$.  Thus 
\begin{align*}  \sum_{\substack{a \geq 0 \\ g_0 \in G}} z^{-a - 1}\br{ \psi^{a} \phi^c_{ g_0}, \phi^c_{g_1},  \ldots,  \phi^c_{  g_n} }^c_{0,n+1} \phi^{c, g_0} 
= z^{1 - n} \phi^c_{\prod_{i = 1}^n g_i }.
\end{align*}  The lemma follows.
\end{proof}

\begin{lemma} \label{l:4.3} The transformation $i_c$ identifies derivatives of the two $J$-functions:
\[
 {i_c} \left( z \frac{\partial}{\partial t^{\jj^c g' }_{0}}  J^{0,\bo}(\bt^0, z) \right)
 =  z \frac{\partial}{\partial t^{g'}_{c}}  J^{c, \bo}(\bt^c, z)
 |_{t^g_{c} = t^g_{0}}.
\]
In particular, 
\[
 {i_c} \left( z \frac{\partial}{\partial t^{ \jj^c}_{0}}  J^{0, \bo}(\bt^0, z) \right)
 = J^{c, \bo}(\bt^c, z)
 |_{t^g_{c} = t^g_{0}}.
\]
\end{lemma}

\begin{proof}
%
%
Observe that 
\[
\begin{split}
  &i_c\left(z \frac{\partial}{\partial t^{\jj^cg '}_0} J^{0, \bo}(\bt^0, z)   \right) \\
 =  &i_c \left(\phi^0_{ \jj^c g '}z + \sum_{g \in G} t^g_0 \phi^0_{ \jj^cg 'g}+ 
   \sum_{\{ a_g \geq 0\}_{g \in G}} z^{\sum 1- a_g}
   \prod_{g \in G}\frac{ (t_0^g)^{a_g} }{ a_g! } \phi^0_{ \jj^c g' \prod_g g^{a_g}} \right)
\\ 
 = &\phi^c_{g '}z + \sum_{g \in G} t^g_0 \phi^c_{ g' g}+ 
   \sum_{\{ a_g \geq 0\}_{g \in G}} z^{\sum 1- a_g}
   \prod_{g \in G}\frac{ (t_0^g)^{a_g} }{ a_g! } \phi^c_{  g' \prod_g g^{a_g}} 
\\
  =& z \frac{\partial}{\partial t^{g'}_c} J^{c, \bo}(\bt^c, z)|_{ \{t^{g}_c=t^{g}_0\}} .   \\
\end{split}
\]
The second statement then follows from the string equation:
\[
z \frac{\partial}{\partial t^{e}_{c}}  J^{c, \bo}(\bt^c, z) = J^{c, \bo}(\bt^c, z).
\]

\end{proof}

\begin{proposition}\label{p:untwistedidentification}
Under the isomorphism $i_c: \sV^{0, \bo} \to \sV^{c, \bo}$, the Lagrangian cones $\sL^{0, \bo}$ and $\sL^{c, \bo}$ are identified.
\end{proposition}

\begin{proof}
%
%

By~\eqref{e:generic} and~\eqref{e:Jgens2}, the Lagrangian cone $\sL^{c, \bo}$ is spanned by the set of derivatives $\{z \partial/\partial t^{g'}_{c} J^{c, \bo}(\bt^c, z)\}_{g' \in G}$. From the above lemma we see that $i_c$ identifies the set of derivatives of $J^{0, \bo}$ with those of $J^{c, \bo}$ and therefore identifies the two Lagrangian cones.
%
%
%
\end{proof}

\subsection{The \mlk theory correspondence}

We are now able to state a relationship between the twisted theories corresponding to different powers $c$ of $\omega_{\cC, \on{log}}$.  As it relates \emph{multiple powers of the log-canonical}, we call it
the \mlk correspondence.
\begin{theorem}[The \mlk correspondence]\label{c:cones}
The isomorphism $i_c: \sV^{0, \bs} \to \sV^{c, \bs}$ identifies the $\bs$-twisted Lagrangian cones $\sL^{0, \bs}$ and $\sL^{c, \bs}$.
\end{theorem}

\begin{proof}
Note that $m_j(\phi^c_{ g}) = m_j( g\jj^c) = m_j(\phi^0_{ g\jj^c})$.  Thus the action of $\Delta^0$ on the subspace spanned by $\phi^0_{ g\jj^c}$ is the same as the action of $\Delta^c$ on the subspace spanned by $i_c(\phi^0_{ g\jj^c}) = \phi^c_g$.  In other words, under the identification given by $i_c$, $\Delta^0$ and $\Delta^c$ are the same symplectic transformation.  By Proposition~\ref{p:untwistedidentification} identifying the untwisted Lagrangian cones, we see that $i_c \sL^{0, \bs} = i_c \Delta^0 \sL^{0, \bo} = \Delta^c \sL^{c, \bo} = \sL^{c, \bs}$.
\end{proof}

\begin{corollary}\label{c:Jrels}
$$z \frac{\partial}{\partial t_0^{\jj^c}}i_{c}\left( J^{0, \bs}(\bt_0, z) \right) = J^{c, \bs}(\bt_c, z)$$ 
%
%
where the change of variables is given by   
$$\langle \bt_c, \phi^c_g\rangle^{c, \bs} = 
\frac{\partial^2}{\partial t^{g\jj^c}_0 \partial t^{\jj^c}_0} \cF_0^{0, \bs}(\bt_0, 0, 0, \ldots) 
.$$
\end{corollary}

\begin{proof}
By \eqref{e:Jgens}, and the above Corollary,  the ruling of the Lagrangian cone $\sL^{c, \bs}$ at $J^{c, \bs}(\bt_c, -z)$ is in fact spanned by the derivatives of $J^{0, \bs}(\bt_0, -z)$.  Thus we have 
\[z \sum_{g \in G} C^g(\bt_0,z) \frac{\partial}{\partial t^g_0} i_{c}\left( J^{0, \bs}(\bt_0, -z) \right) = J^{c, \bs}(\bt_c, -z)\]
for some functions $C^g( \bt_0, z)$ and some change of variables between $\bt_c$ and $\bt_0$.  Equating coefficients of $z$ on either side yields $C^{\jj^c}(\bt_0, z) = 1$ and all other $C^g(\bt_0, z)$ equal zero.  The change of variables is obtained by then equating coefficients of $z^0$.
\end{proof}

\begin{remark}
The above results should be viewed as akin to quantum Serre duality as given in \cite{G1}, \cite{CG} and \cite{Ts}, and summarized here in Theorem~\ref{t:qsd}.  Indeed comparing Corollary~\ref{c:Jrels} with Theorem~\ref{t:qsd}, one sees that they take an almost identical form.
\end{remark}

\begin{remark}
By applying Teleman's proof (\cite{Te}) of Givental's conjecture  for semi-simple Frobenius manifolds (\cite{G4}) to the above untwisted theories, one may deduce a higher genus correspondence between untwisted theories.
Combining this with Theorem~\ref{t:symplectictransformation}, one obtains a higher genus analogue of the correspondence of Theorem~\ref{c:cones}.  This will relate the total genus descendant potentials of the twisted theories via a quantized symplectic operator.  We leave the details to the reader.
\end{remark}
%
%

\subsection{Implications to local GW and FJRW theory}
In this subsection, we apply the \mlk correspondence to prove a relationship between  local GW theory and the FJRW theory.

Let $(Q, G)$ be a Landau--Ginzburg pair where $Q$ is Fermat.  
Recall that in this case $\bar{d} = d$
and we have the relationship mentioned above (Section~\ref{s:rtoi}) between the $\bs$-twisted theories and both local GW theory and FJRW theory.  
In this section we fix the specialization of the $\bs$ parameter to 
\begin{equation}\label{e:spec}
s^j_0 = -\ln(-\lambda_j)\text{, } s^j_k = (k-1)!/\lambda_j^k \text{ for } 1 \leq j \leq N, \: k > 0.
\end{equation}  Recall that under this specialization $\bs(V) = 1/e_{\CC^*}(V)$. We will still refer to these as $\bs$-twisted invariants, where it is understood that we have specialized the $s^j_k$ as above.

\subsubsection{Local GW theory}
In the case $c=0$, specializing the $\bs$-twisted $W^c$ invariants as above recovers the local GW invariants of 
$[\CC^N/G]$ after multiplying by a factor of $d^N$.  
The pairing also differs by this factor.  Consider the symplectic transformation 
\[
 \sV^{0, \bs} \to \sV^{[\CC^N/G]}
\]
induced by
\[
 \phi^0_{ g} \mapsto \frac 1 {\sqrt{d^N}} \ii_g.
\]
Under this transformation, we have the equality 
\[
\bigg\langle \frac{\phi^0_{ g}}{z - \psi},  \phi^0_{ g_1},\ldots, \phi^0_{ g_n} \bigg\rangle_{0, n+1}^{0, \bs} \phi^{0, g} = \frac 1 {\sqrt{d^N}} \bigg\langle   \frac{\ii_g}{z - \psi},  \ii_{g_1},\ldots, \ii_{g_n} \bigg\rangle_{0, n+1}^{[\CC^N/G]} \ii^g,
\]
where upper indices denote dual elements with respect to the given basis.  Therefore the respective $J$-functions also differ by an overall factor.  
We obtain the following lemma.

\begin{lemma}  

After specializing the $s^j_k$ as in \eqref{e:spec},
  \[\sL^{0, \bs} = \sL^{[\CC^N/G]}\]
  under the identification  $\phi^0_{ g} \mapsto \frac 1 {\sqrt{d^N}} \ii_g$.
\end{lemma}
  
 \begin{proof}
 Let $\bt   = \sum_{g \in G} {t^g} \ii_g$.  By the above we see that 
 \begin{equation*}
 J^{0, \bs}(\bt_0, z) = \frac 1 {\sqrt{d^N}} J^{[\CC^N/G]}(\bt , z)
\end{equation*}
after the change of variables $t^g_0 = {t^g} $.  In particular they generate the same Lagrangian cone. 
 \end{proof}
 
\subsubsection{FJRW theory}  Let $\bs$ be as in \eqref{e:spec}.
A similar identification results between the $\bs$-twisted $J$-function at $c=1$ and the corresponding FJRW $J$-function.  In this case we must account for the factor of plus or minus one, determined by the parity of $\chi( \oplus_{j=1}^N \cL_j)$ (see~\eqref{e:2.2.1}).  Adjusting our specialization of $s^j_0 = -\ln(-\lambda_j)$ to ${s^j_0} ' = -\ln(-\lambda_j) + \pi \sqrt{-1}$ as in~\eqref{e:bs2} will have the effect of modifying our twisted invariants by this sign.  This alters the symplectic transformation of Theorem~\ref{t:symplectictransformation} by 
\begin{align*} & \bigoplus_{g \in G} \prod_{j = 1}^N \exp \left( B_1(m_j(\phi^1_{g})) \pi \sqrt{-1} \right)\\
=& \bigoplus_{g \in G}  \exp \left( B_1\left(\sum_{j=1}^N m_j(\phi^1_{g})\right) \pi \sqrt{-1} \right)\\
\sim&   \bigoplus_{g \in G} \exp \left( \sum_{j=1}^N m_j(g\jj) \pi \sqrt{-1} \right)\\
\sim & \bigoplus_{g \in G} (-1)^{\sum_{j=1}^N m_j(g\jj)}
\end{align*}
where $\sim$ means equal up to a constant factor (which will not effect the Lagrangian cone).  Define \[ \Delta ' := \bigoplus_{g \in G} (-1)^{\sum_{j=1}^N m_j(g\jj)}.\]

\begin{lemma}\label{l:4.15}  Given the specialization of $s^j_k$ as in \eqref{e:spec}, let $F^{1, \bs}$ be a point on $\sL^{1, \bs}$ such that the nonequivariant limit $\lim_{\lambda \mapsto 0} F^{1, \bs}$ is both well defined and supported in the subset $\left( \oplus_{g, \in \hat G} \CC \cdot \phi^1_g\right)[z] \oplus \sV^{1, \bs}_{-} \subset
\sV^{1, \bs}$.
Then $\lim_{\lambda \mapsto 0} \Delta '  ( F^{1, \bs})$ lies in $\sL^{(Q,G)}$ under the identification
$ \phi^1_{ g} \mapsto \frac 1 {\sqrt{d^N}} \varphi_g$.
\end{lemma}

\begin{proof} The symplectic transformation $\Delta '$ maps $\sL^{1, \bs}$ to $\sL^{1, \bs '}$ where ${s^j_k} ' = s^j_k$ for $k > 0$ and ${s^j_0} ' = s^j_0 + \pi \sqrt{-1}$ as in~\eqref{e:bs2}.  By \eqref{e:conedescr},  $
\lim_{\lambda \mapsto 0} \Delta '  ( F^{1, \bs})$ may be written as
\[
-z \phi^1_0 + \sum_{\substack{k \geq 0 \\ g \in \hat G}} t^g_k \phi^1_g z^k + \lim_{\lambda \mapsto 0} \sum_{\substack{a_1, \ldots , a_n,  a\geq 0 \\ g_1, \ldots , g_n \in \hat G \\ g\in G}} \frac{t^{g_1}_{a_1}\cdots t^{g_n}_{a_n}}{n!(-z)^{a+1}}\langle \psi^a\phi^1_g ,\psi^{a_1}\phi^1_{g_1},\dots,\psi^{a_n}\phi^1_{g_n}\rangle_{0, n+1}^{1, \bs '}\phi^{1,g}.
\]
Recall Lemma~\ref{l:1.10},  which with the specialization $\bs '$ implies that $\phi^{1,g} = d^N \prod_{j = 1}^N( \lambda_j)^{\lfloor 1- m_j(g \jj ) \rfloor} \phi_{g^{-1}\jj^{-2}}$.  Therefore, for $g \notin \hat G$, the power of $\lambda$ in this expression for $\phi^{1, g}$ is positive.  Note also that $g \in \hat G$ if and only if $g^{-1}\jj^{-2} \in \hat G$.  By~\eqref{e:degree}, one calculates that $\langle \psi^a\phi^1_g ,\psi^{a_1}\phi^1_{g_1},\dots,\psi^{a_n}\phi^1_{g_n}\rangle_{0, n+1}^{1, \bs '} \in \CC[\lambda]$.  Thus in the non-equivariant limit of the above expression all terms containing the insertion $\phi^1_g$ for $g \notin \hat G$ vanish.

After applying $\phi^1_{ g} \mapsto \frac 1 {\sqrt{d^N}} \varphi_g $ and recalling \eqref{e:2.2.1}, the above expression becomes
\[
 \begin{split}
 \frac{1}{\sqrt{ d^N}} \bigg(  &-z \varphi_0 +  \sum_{\substack{k \geq 0 \\ g \in \hat G}} t^g_k \varphi_g z^k \\
 &+ \sum_{\substack{a_1, \ldots , a_n,  a\geq 0 \\ g_1, \ldots , g_n, g \in \hat G}} \frac{t^{g_1}_{a_1}\cdots t^{g_n}_{a_n}}{n!(-z)^{a+1}}\langle \psi^a\varphi_g ,\psi^{a_1}\varphi_{g_1},\dots,\psi^{a_n}\varphi_{g_n}\rangle_{0, n+1}^{(Q,G)}\varphi^g \bigg),
 \end{split}
\]
which gives a point on $\sL^{(Q,G)}$.

\end{proof}
%
%

\subsubsection{The correspondence}\label{s:5.2.3}

Let $\Delta^{\circ}$ denote the symplectic transformation given by 
\begin{equation*}
\Delta^\circ: \ii_g \mapsto (-1)^{\sum_{j=1}^N m_j(g)}\varphi_{g\jj^{-1}}.
\end{equation*}  
%
The previous two lemmas together with Corollary~\ref{c:Jrels} allow us to determine the FJRW $J$-function from that of $[\CC^N/G]$.  

\begin{theorem}[the Landau--Ginzburg/local GW correspondence]\label{t:gw/fjrw}  We have the relationship \[ \lim_{\lambda \mapsto 0}\Delta^\circ \left( z \frac{\partial}{\partial t^{j}} \left( J^{[\CC^N/G]}(\bm{t}, z) \right)\right) = J^{(Q,G)}(\bm{t '}, z)\] with the substitution given by   \[\langle \bm{t '}, \varphi_g \rangle = \lim_{\lambda \mapsto 0} 
\frac{\partial^2}{\partial t^{gj} \partial t^j} \cF_0^{[\CC^N/G]}(\bm{t}, 0, 0, \ldots) 
\] for any $g \in \hat G$.
\end{theorem}

\begin{proof}  
Consider the function $z \frac{\partial}{\partial t^{j}} \left( J^{[\CC^N/G]}(\bm{t}, z) \right) = z \ii_\jj + \sO(1)$.  The terms with non-positive $z$-coefficient are of the form
\begin{equation*}
z^{-k}\br{\ii_g \psi^k, \ii_\jj, \bt, \ldots, \bt }\prod_{j = 1}^N(- \lambda_j)^{  \lfloor 1- m_j(g \jj ) \rfloor } \ii_{g^{-1}}.
\end{equation*}
Due to the insertion of $\ii_\jj$, the universal line bundles over the relevant moduli space have negative degree, and thus $\br{\ii_g \psi^k, \ii_\jj, \bt, \ldots, \bt } \in \CC[\lambda]$.  We see that the coefficient of $\ii_g$ in $z \frac{\partial}{\partial t^{j}} \left( J^{[\CC^N/G]}(\bm{t}, z) \right)$ is a $\CC[\lambda]$-multiple of $\lambda^{N_g}$.
Therefore, $i_1\left( z \frac{\partial}{\partial t^{j}} \left( J^{[\CC^N/G]}(\bm{t}, z) \right)\right)$ satisfies the hypotheses of $F^{1, \bs}$ from the previous lemma.  We conclude that 
\[
 \lim_{\lambda \mapsto 0} \Delta^\circ \left( -z \frac{\partial}{\partial t^{j}} \left( J^{[\CC^N/G]}(\bm{t}, -z) \right)\right) = \lim_{\lambda \mapsto 0} \Delta ' \circ i_1\left(- z \frac{\partial}{\partial t^{j}} \left( J^{[\CC^N/G]}(\bm{t}, -z) \right)\right)
\]
lies on $\sL^{(Q,G)}$.  The result then follows by examining the coefficients of $z^1$ and $z^0$ in the above expression.
\end{proof}

The following more general statement will prove useful for applications.  The proof is the same argument as above.

\begin{theorem}\label{t:general1}
Let $F^{[\CC^N/G]}(\bm{t}, z)$ be a function lying on $\sL^{[\CC^N/G]}$ 
such that the projection of the non-equivariant limit $\lim_{\lambda \mapsto 0} F^{[\CC^N/G]}(\bm t, z)$ to $\sV^{[\CC^N/G]}_+$ is both well defined and is supported in the span of $\ii_g$ such that $g$ fixes only the origin in $\CC^N$.
Then $\lim_{\lambda \mapsto 0} \Delta^\circ ( F^{[\CC^N/G]}(\bm t, z))$ lies on $\sL^{(Q,G)}$ .
\end{theorem}
%


\begin{remark}
Theorem~\ref{t:gw/fjrw} should extend more generally to the setting of \emph{hybrid theories}, where the moduli spaces $W_{g, n, G}(\cX)$ parameterize stable maps from curves into a target $\cX$ together with roots of certain universal bundles.  In this setting the $c=0$ case would correspond to local GW theory over $\cX$ and the $c=1$ case to a hybrid theory.  See \cite{Cl} and also \cite{CZ} for more details on this setting. 
\end{remark}

\subsection{Quantum Serre duality}\label{s:qsd}

Here we recall the statement of quantum Serre duality.  The purpose of this section is two-fold.  First, we wish to emphasize the analogy between quantum Serre duality and the \mlk
correspondence given above.  Second, we will use these results in the next section to relate the crepant transform conjecture to the LG/CY correspondence.

Let $\cX$ be a smooth projective orbifold and let $E \to \cX$ be a vector bundle over $\cX$ which is pulled back from the coarse underlying space.  Given an invertible multiplicative characteristic class 
\[ 
  \bs: V \mapsto \exp\left( \sum_{k \ge 0} s_k \ch_k(V)\right),
\]
we may define the $\bs$-twisted GW invariants of $\cX$ in a manner akin to \eqref{e:twistedinvariant} (see \cite{CG} for details).  We will denote these invariants and their corresponding generating functions with the superscript $E, \bs$.

Quantum Serre duality gives a relation between invariants twisted with respect to the vector bundle $E$ and those twisted with respect to the dual bundle $E^\vee$.  The main statement in genus zero is given below.  This is Corollary 10 of \cite{CG}, and follows in the orbifold case from Theorem 6.1.1 in \cite{Ts}.  Let $\{\gamma_i\}_{i \in I}$ be a basis for $H^*_{CR}(\cX)$.  
Let $\bs^*: V \mapsto \exp\left( \sum_{k \ge 0} (-1)^{k+1} s_k \ch_k(V)\right)$, so that $\bs^*(E^\vee) = \frac{1}{\bs(E)}$.  
\begin{theorem}[Quantum Serre duality]\label{t:qsd}  Define the (symplectic) transformation $i_{E^\vee}: \sV^{E^\vee, \bs^*} \to \sV^{E, \bs} $ by $\gamma_i \mapsto \gamma_i/\bs(E)$.  Then
$i_{E^\vee}( \sL^{E^\vee, \bs^*}) = \sL^{E, \bs}$.  Furthermore we have the identification
$$z \frac{\partial}{\partial t^{\bs(E)}_{E^\vee}}i_{E^\vee}\left( J^{E^\vee, \bs^*}(\bt_{E^\vee}, z) \right) = J^{E, \bs}(\bt_E, z)$$ 
where the change of variables is given by   
$$\langle \bt_E, \gamma_i \rangle = \sum_{i \in I} \frac{\partial^2}{\partial t^i_{E^\vee} \partial t^{\bs(E)}_{E^\vee}} \cF_0^{E^\vee, \bs^*}(\bt_{E^\vee}, 0, 0, \ldots) .$$
\end{theorem}

\subsubsection{Complete intersections and local invariants}\label{sss:lici}

Consider the special case where $E$ is a direct sum of convex line bundles ($H^1_{CR}(\cC, f^*(E)) = 0$ for all maps $f$ from a curve $\cC$ into $\cX$) and $\bs$ is the equivariant Euler characteristic:
\begin{equation}\label{e:specCI}
s_0 = \ln(\lambda)\text{, } s_k = (-1)^k(k-1)!/\lambda^k \text{ for } k > 0.
\end{equation}
In this case, the genus-zero $\bs$-twisted invariants with respect to $E$ are related to invariants of the hypersurface $\cZ$ cut out by a generic section of $E$ by the so-called \emph{quantum Lefschetz principle} (\cite{CG}, \cite{Ts}).  
Coates has recently rephrased this relationship in terms of Lagrangian cones.  Let $i: \cZ \to \cX$ denote the inclusion.  
\begin{theorem}[\cite{Co}]\label{t:qsd2} Assume that the vector bundle $E$ is pulled back from a vector bundle over the coarse space of $\cX$.  After specializing the $s_k$ as above, let $F^{E, \bs}$ be a point in $\sL^{E, \bs}$ with a well defined non-equivariant limit.  Then $\lim_{\lambda \mapsto 0} i^* (F^{E, \bs})$ lies on $\sL^\cZ$.
\end{theorem}

\begin{remark}
Although the above theorem was proven only for the case of $\cX$ a smooth variety, the proof extends to orbifolds provided we assume that $E$ is pulled back from a vector bundle $|E| \to |\cX|$ over the coarse underlying space of $\cX$.
\end{remark}

On the other hand, if we specialize to 
\[
s_0 ' = -\ln(-\lambda)\text{, } s_k ' = (k-1)!/\lambda^k \text{ for } k > 0
\]
as in \eqref{e:spec}, the $\bs^*$-twisted invariants with respect to $E^\vee$ give the local invariants of the total space of $E^\vee$.  For $\gamma_1, \ldots, \gamma_n \in H^*_{CR}(\cX)$,
\begin{equation}  \br{ \psi^{a_1} \gamma_1, \ldots, \psi^{a_n}  \gamma_n }^{\text{Tot}(E^\vee)}_{\h,n, d} = \br{ \psi^{a_1} \gamma_1, \ldots, \psi^{a_n} \gamma_n }^{E^\vee, \bs^*}_{\h,n, d}.
\end{equation}

Theorem~\ref{t:qsd} implies a relation between the local invariants of $\text{Tot}(E^\vee)$ and the invariants of the hypersurface $\cZ$.  For our purposes, it is most useful to phrase the relationship in a manner analogous to Theorem~\ref{t:general1}.  Let $\Delta^\diamond$ denote the symplectic transformation
\begin{equation*}
\Delta^\diamond: \phi_i \mapsto e^{\pi \sqrt{-1}c_1(E)/z}\frac{(-1)^{\text{rk}(E)}}{e_{\CC^*}(E)}\phi_i.
\end{equation*}

\begin{theorem}\label{t:general2}
Let $F^{Tot(E^\vee)}(\bm t, z)$ be a function lying on $\sL^{Tot(E^\vee)}$. Assume further that $F^{Tot(E^\vee)}(\bm t, z)$ takes the form
\[F^{Tot(E^\vee)}(\bm t, z) = e_{\CC^*}(E) \widetilde{ F}^{Tot(E^\vee)}(\bm t, z),\] and that
$\widetilde F^{Tot(E^\vee)}(\bm t, z)$ has
a well defined non-equivariant limit.  
Then $\lim_{\lambda \mapsto 0} i^* \circ \Delta^\diamond (F^{Tot(E^\vee)}(\bm t, z))$ lies on $\sL^\cZ$.
\end{theorem}

\begin{proof} 
The symplectic transformation $\Delta^\diamond$ may be written as $\Delta '' \circ i_{E^\vee}$, where $i_{E^\vee}$ is as in Theorem~\ref{t:qsd}, and $\Delta '' = e^{\pi \sqrt{-1} c_1(E)/z}$. The map $\Delta ''$ compensates for the fact that with our given specializations, $s_0 '$ does not equal $-s_0$ as in the relationship between $s_0^*$ and $s_0$ in Theorem~\ref{t:qsd}, but rather $s_0 ' = -s_0 - \sqrt{-1} \pi $ (see the remark after Theorem 1' in \cite{CG} for details).  By Theorem~\ref{t:qsd}, $\Delta^\diamond (F^{Tot(E^\vee)}(\bm t, z))$ lies in $\sL^{E, \bs}$ with the specialization \eqref{e:specCI}.  The specific assumptions on $F^{Tot(E^\vee)}(\bm t, z)$ guarantee that $\Delta^\diamond (F^{Tot(E^\vee)}(\bm t, z))$ has a well defined non-equivariant limit.  Theorem~\ref{t:qsd2} then implies the result.
\end{proof}

\section{The CTC and the LG/CY correspondence}\label{s:ctclgcy}

In this section we give an application of Theorem~\ref{t:gw/fjrw}.  In particular we use it together with its analogue, quantum Serre duality (Theorem~\ref{t:general2}), to relate two well known conjectures from Gromov--Witten theory.  We show that in genus zero, the well known \emph{crepant transformation conjecture} (also known as the crepant resolution conjecture) from \cite{CCIT, CR} implies the more recent \emph{LG/CY correspondence} of \cite{ChR}.

\subsection{The Conjectures}
Let $(Q, G)$ be an LG pair with $Q = \sum_{j=1}^N x_j^{d/c_j}$ a Fermat polynomial and $G$ an admissible subgroup of $SL_N(\CC)$.  We assume always that $\gcd(c_1, \ldots, c_N) = 1$.  Let $\bar{G}$ denote the quotient $G/\langle \jj \rangle$.  Note that $Q$ may be viewed as a homogeneous function on the stack quotient $\PP(G) = [\PP(c_1, \ldots, c_N)/ \bar{G}]$.  
\begin{definition}\label{d:cyc}
We say a quasi-homogeneous polynomial satisfies the \emph{Calabi--Yau condition} if $\sum_{i = 1}^N c_i = d$.  This is equivalent to requiring that $Q$ give a global section of the anticanonical bundle of  $\PP(G)$, which, by the adjunction formula, implies that $\{Q = 0\}$ defines a Calabi--Yau variety.
\end{definition}
Assume from here forward that $Q$ satisfies the Calabi--Yau condition. Let $\cX$ denote the stack quotient $[\CC^N/ G]$ and let $\cY$ denote the total space of the canonical bundle $K = K_{\PP(G)}$.

\begin{lemma}
The space $\cY$ gives a toric crepant partial resolution of $\cX$.
\end{lemma}

\begin{proof}
Let $M \cong \ZZ^N$ denote a lattice and let $\Sigma \subset M$ be a fan such that $X_\Sigma = \PP(G)$.  Let $p_1, \ldots, p_N \in M$ be the primitive generators of the $N$ rays of $\Sigma$. The cones of $\Sigma$ are exactly those whose extremal rays are generated by $\{p_{j_1}, \ldots, p_{j_k}\}$ where $\{j_1, \ldots, j_k\}$ is a strict subset of $\{1, \ldots,  N\}$.  Abusing notation, we will identify a cone with its ray generators.

%

Let $\widetilde M$ denote the augmented lattice $M \oplus \ZZ$, and define $\tilde p_j := (p_j, 1) \in \widetilde M$.  Define $\Sigma '$ as the fan in $\widetilde M$ consisting of the cones $\{\tilde p_{j_1}, \ldots, \tilde p_{j_k}\}$ for $\{j_1, \ldots, j_k\}$ \emph{any} subset of $[[1, N]]$.  Define $\widetilde \Sigma$ as the star subdivision of $\Sigma '$ after adding the ray generated by $(\mathbf{0}, 1)$ where $\mathbf{0}$ is the origin in $M$.  One may check using simple toric arguments that $\cX$ is equal to the toric stack $X_{\Sigma '}$, and $\cY$ is $X_{\widetilde \Sigma }$.  It is apparent from this description that $\cY$ is a toric partial resolution of $\cX$.  Furthermore note that all ray generators of $\Sigma '$ are at height one in the augmented coordinate, as is the added ray $(\mathbf{0}, 1)$ defining the resolution.  This implies that $\cY \to |\cX|$ is crepant.
\end{proof}

The inertia orbifold $I\cX$ is a disjoint union of components $\cX_g$ indexed by $g \in G$.  There is a natural choice of basis for the equivariant cohomology of $\cX$ given by $\{\ii_g\}_g \in G$, where $\ii_g$ is the fundamental class of $\cX_g$.  

%
The components of the inertia orbifold $I_\cY$ are indexed by those $g \in G$ which fix a positive-dimensional subspace of $\CC^N$, i.e. $N_g >0$.
For notational convenience we will write $I \cY = \coprod_{g \in G} \cY_g$, with the understanding that $\cY_g$ is empty unless $N_g > 0$.
An equivariant basis for the Chen--Ruan cohomology of $\cY$ is given by 
\[\cup_{g \in G} \{ \tilde \ii_g, \tilde \ii_g H, \ldots, \tilde \ii_g H^{(N_g - 1)}\},
\]
where $\tilde \ii_g$ is the fundamental class of $\cY_g$ and $\tilde \ii_g H^k$ denotes the pullback of the $k$th power of the hyperplane class from the course space of $\cY_g$.  Here again we use the convention that $\ii_g$ is zero if $\cY_g$ is empty.

Gromov--Witten theory for local toric targets is usually defined in terms of equivariant cohomology.  In our case we use a $\CC^*$-action on $\cX$ with weight $-c_j$ on the $j$th component, in other words $\lambda_j = c_j\lambda$.  The corresponding action on $\cY$ is by multiplication in the fiber direction with character $-d\lambda$. 
We now give (a refined version of) the genus zero crepant transformation conjecture.  Let $\sL^{\cX} \subset \sV^{\cX}$ and $\sL^{\cY} \subset \sV^{\cY}$ denote the Lagrangian cones corresponding to the equivariant GW theory of $\cX$ and $\cY$ respectively.  We distinguish two coordinates in the respective $J$-functions.  Let $t = t^\jj$ denote the dual coordinate to $\ii_\jj$ in $H^*_{CR}(\cX)$.  Let $q$ denote the \emph{exponential} of the dual coordinate to the hypersurface $H$ in $H^*(\cY) \subset H^*_{CR}(\cY)$.  By the divisor equation, the function $J^\cY$ (and therefore the Lagrangian cone $\sL^\cY$) is a well defined function of $q$ (\cite{AGV}). 
Let us assume further that there exists a function $I^\cY( \bt, z)$ which generates $\sL^\cY$ in the sense of~\eqref{e:Jgens} and is in fact analytic in a neighborhood of $q = 0$.  
Then
via the change of variables 
\[ q = t^{-d}\]
and analytic continuation, we can view $I^\cY( \bt, z)$ as a function of $t$. Thus it makes sense to analytic continue $\sL^\cY$ from $q = 0$ to $t = 0$.

\begin{conjecture}[The crepant transformation conjecture for $\cY \dasharrow \cX$, \cite{CIT, CR}]\label{c:CTC}  The analytic continuation of $\sL^\cY$ converges in a neighborhood of $t = 0$, and
there exists a symplectic transformation $\UU: \sV^\cX \to \sV^\cY$ which identifies $\sL^{\cX}$ with the analytic continuation of $\sL^{\cY}$.  
\end{conjecture}

In our case we deal with local targets, here one may refine the above conjecture to take into account the equivariant nature of the theory.

\begin{conjecture}[The refined crepant transformation conjecture]\label{c:rCTC}
Conjecture~\ref{c:CTC} holds.  In addition 
the following conditions are satisfied:
\begin{enumerate}
\item \label{i:c1} $\UU$ has coefficients in $\CC[\lambda, z, z^{-1}]$.  In the non-equivariant limit, $\UU$ restricts to an isomorphism between the subspaces of $\sV^\cX_c \subset \sV^\cX$ and $\sV^\cY_c \subset \sV^\cY$ spanned by classes of compact support.
\item \label{i:c2}  
$$\UU( \ii_g) = C_0( \lambda) \tilde \ii_g + \sum_{b = 1}^{d - 1} (\lambda + H) \cdot C_b(\lambda) \tilde \ii_{g \jj^b}$$
where $C_b(\lambda) \in H^*(\cY)[\lambda]((z^{-1}))$.  In particular,
the restriction, $\UU_c$, of $\UU$ to $\sV_c^\cX$
has image in the $\CC((z^{-1}))$-span of $(\lambda + H)\cdot H^*_{CR}(\cY)[\lambda]$.
%
\end{enumerate}
\end{conjecture}

\begin{remark} 
Conditions \eqref{i:c1} and \eqref{i:c2} above are very natural.
It is generally believed that the symplectic transformation $\UU$ should be induced by a Fourier--Mukai transform between equivariant $K$-groups, in the sense of \cite{Iri}.  In this case, $\UU$ will automatically be symplectic, because the Fourier--Mukai transform is a category equivalence and preserves the categorical Euler pairing.  Furthermore the Fourier--Mukai transform has a nonequivariant limit and preserves the compactly supported part of the $K$-groups, which induces the corresponding properties in $\UU$.  See \cite{Iri} for more details.  In the next section we give further evidence for these conditions.
%
\end{remark}

The Landau--Ginzburg/Calabi--Yau (LG/CY) correspondence takes a similar form to Conjecture~\ref{c:CTC}.
Given an LG pair $(Q, G)$ as in the previous section, let $\cZ$ denote the Calabi--Yau variety $\{Q = 0\} \subset \PP(G)$.  Let $i:\cZ \to \PP(G)$ denote the inclusion.  The LG/CY correspondence relates the FJRW theory of $(Q,G)$ to the GW theory of $\cZ$ in a similar fashion to the crepant transformation conjecture.  In particular, in genus zero the conjecture states that there is a symplectic transformation identifying the respective Lagrangian cones.  

\begin{conjecture}[The LG/CY correspondence for $(Q,G)$]\label{c:LGCY}
There exists a symplectic transformation $\VV: \sV^{(Q, G)} \to \sV^{\cZ}$ which identifies $\sL^{(Q,G)}$ with the analytic continuation of $\sL^{\cZ}$.  
\end{conjecture}


\subsection{rCTC implies LG/CY}

In this section we give an explanation of the similarity between these two correspondences.  Namely we show that the refined crepant transformation conjecture implies the LG/CY correspondence.  

\begin{lemma}\label{l:5.6} Assuming Conjecture~\ref{c:rCTC}, define the map $\VV$ by
\[\VV := \left( i^*\circ \Delta^\diamond \circ \UU_c \circ (\Delta^\circ)^{-1} \right)|_{\lambda = 0}.\]  Then $\VV$ is symplectic.
\end{lemma}

\begin{proof}
We will use conditions~\eqref{i:c1} and~\eqref{i:c2} from Conjecture~\ref{c:rCTC}.  First, for any \emph{nonequivariant} compactly supported class $\alpha \in H^*_{CR, c}(\cX)((z^{-1})) \subset H^*_{CR}(\cX)((z^{-1}))$, $\UU_c(\alpha) = \UU(\alpha)$ may be written as 
\[(\lambda + H)\cdot \sum_{\substack{g \in G\\ 0 \leq k \leq N_g-1}} C^\alpha_{g,k}(\lambda) \ii_g H^k,\] where $C^\alpha_{g, k}(\lambda) \in \CC[\lambda]$.  Note that in the non-equivariant limit of $\UU_c(\alpha)$, the terms of the form $(\lambda + H) C^\alpha_{g, N_g - 1}(\lambda) \ii_g H^{N_g-1}$ vanish.  By condition~\eqref{i:c1}, the map 
\[ (\UU_c)|_{\lambda = 0}: \alpha \mapsto H \cdot \sum_{\substack{g \in G\\ 0 \leq k \leq N_g-2}} C^\alpha_{g, k}(0) i^*( \ii_g H^k)\] 
is an isomorphism.

For $\alpha, \beta \in H^*_{CR, c}(\cX)((z^{-1}))$, the pairing $\br{ \alpha, \beta } \in \CC$.  Since $\UU$ and $\Delta^\diamond$ are symplectic this implies that 
\[
 \br{\Delta^\diamond \circ \UU(\alpha), \Delta^\diamond \circ \UU(\beta)} = \br{ \Delta^\diamond \circ \UU_c(\alpha), \Delta^\diamond \circ \UU_c(\beta)} \in \CC
\] and thus 
\[ \br{ (\Delta^\diamond \circ \UU_c)|_{\lambda = 0}(\alpha),   (\Delta^\diamond \circ \UU_c)|_{\lambda = 0}(\beta)} = \lim_{\lambda \mapsto 0}\br{ \Delta^\diamond \circ \UU_c(\alpha), \Delta^\diamond \circ \UU_c(\beta)} = \br{\alpha, \beta}.
\]  
Therefore 
\[
(\Delta^\diamond \circ \UU_c)|_{\lambda = 0}: \alpha \mapsto \sum_{\substack{g \in G\\ 0 \leq k \leq N_g-1}} C^\alpha_{g, k}(0) i^*( \ii_g H^k)
\] 
is in fact a symplectic isomorphism.  On the other hand, examining the pairing given by the twisted theory of $(-K, \bs)$, where $\bs$ is as in \eqref{e:specCI}, we see that when terms of the form $\ii_g H^{N_g-1}$ are paired with elements of $H^*_{CR}(\cY)[\lambda]$, the result lies in $\sO(\lambda)$.  Therefore these terms do not contribute to the pairing $\br{ (\Delta^\diamond \circ \UU_c)|_{\lambda = 0}(\alpha),   (\Delta^\diamond \circ \UU_c)|_{\lambda = 0}(\beta)} $ for $\alpha, \beta \in H^*_{CR, c}(\cX)((z^{-1}))$.  So in fact we conclude that the map \[\alpha \mapsto \sum_{\substack{g \in G\\ 0 \leq k \leq N_g-2}} C^\alpha_{g, k}(0) i^*( \ii_g H^k)\] is symplectic.

Note that $\lim_{\lambda \mapsto 0} i^*\circ \Delta^\diamond \circ \UU_c (\alpha)$ is equal to
\[i^* \Big(\sum_{\substack{g \in G\\ 0 \leq k \leq N_g-2}} C^\alpha_{g, k}(0)( \ii_g H^k)\Big).\]  
Since $i^*: \sV^{-K, \bs} \to \sV^\cZ$ is a symplectic isomorphism when restricted to the span of 
$\left\{ \ii_g H^k\right\}_{{g \in G, \, 0 \leq k \leq N_g-2}}$,
the map $\alpha \mapsto \lim_{\lambda \mapsto 0} i^*\circ \Delta^\diamond \circ \UU_c (\alpha)$ will be a symplectic isomorphism.  

$\Delta^\circ$ is a symplectic isomorphism when restricted to the span of elements of compact support.  Thus $\VV$, defined as the composition of the above map with $(\Delta^\circ)^{-1}$, is as well.

\end{proof}

\begin{lemma}\label{l:good functions} Assume there exists a function $I^{\cX}(\bt, z)$ lying on $\sL^\cX$ such that for any group element $g$ which fixes more than the origin, there is a power of $\lambda^{N_g}$ in the $\ii_g$-coefficient of $\frac{\partial}{\partial t} I^{\cX}(\bt, z)$.
Then the function 
\[I^{(Q, G)}(\bt, z): = \lim_{\lambda \mapsto 0} \Delta^\circ \left( z \frac{\partial}{\partial t} \left( I^{\cX}(\bm{t}, z) \right) \right)
\] lies on $\sL^{(Q,G)}$, and 
the symplectic transformation $\VV$ maps $I^{(Q, G)}(\bt, z)$  to the analytic continuation of $\sL^\cZ$.
\end{lemma}
\begin{proof}
Note that by assumption 
we can apply Theorem~\ref{t:general1} to deduce that $I^{(Q, G)}(\bt, z)$ lies on $\sL^{(Q,G)}$.

To prove the second part of the lemma, we first claim that 
\[
\lim_{\lambda \mapsto 0} i^* \circ  \Delta^\diamond \circ \UU \left(z \frac{\partial}{\partial t} \left( I^{\cX}(\bm{t}, z) \right) \right)
\] 
lies on the analytic continuation of $\sL^\cZ$.  

To see this let $\widetilde{ I^{\cY}(\bt, z)}$ denote $\UU(I^{\cX}(\bt, z))$.  Conjecture~\ref{c:rCTC} implies that $\widetilde{ I^{\cY}(\bt, z)}$ lies on the analytic continuation of $\sL^\cY$, and therefore by \eqref{e:overruled} so does $z \frac{\partial}{\partial t} \widetilde{ I^{\cY}(\bt, z)}= \UU(z \frac{\partial}{\partial t}I^{\cX}(\bt, z))$.
Therefore the strategy is to show that $z \frac{\partial}{\partial t} \widetilde{ I^{\cY}(\bt, z)}$ satisfies the conditions of Theorem~\ref{t:general2}, i.e. that $z \frac{\partial}{\partial t} \widetilde{ I^{\cY}(\bt, z)}$ may by written as \[e_{\CC^*}(-K) \widetilde{F} (\bt, z),\] where $\widetilde{F}(\bt, z)$ has a well defined non-equivariant limit.  
For each $g \in G$, we will show that $\UU$ maps the part of $z \frac{\partial}{\partial t} I^{\cX}(\bt, z)$
supported on $\cX_g$ to something divisible by $e_{\CC^*}(-K) = d(\lambda + H)$.
%
%
For $g$ such that $N_g = 0$, the statement  follows immediately by condition~\eqref{i:c2} of Conjecture~\ref{c:rCTC} because $\tilde \ii_g = 0$.  For $g$ such that $N_g > 0$, $\tilde \ii_g \neq 0$, but nevertheless
we compute 
\[
\frac{\lambda^{N_g} \tilde \ii_g}{(\lambda + H)} = \sum_{i = 0}^{N_g-1} \lambda^{N_g-1 - i}(-H)^i \ \tilde \ii_g.
\]  
So combining condition~\eqref{i:c2} of Conjecture~\ref{c:rCTC} with the assumptions of the lemma implies the claim.

Given a function $I^{\cX}(\bm{t}, z) $ satisfying the assumptions listed, we have shown that the corresponding $I^{(Q, G)}(\bt, z) $ lies in
$\sL^{(Q, G)}$ and that 
$\lim_{\lambda \mapsto 0} i^* \circ  \Delta^\diamond \circ \UU \left(z \frac{\partial}{\partial t} \left( I^{\cX}(\bm{t}, z) \right) \right)$ lies in the analytic continuation $\widetilde{\sL^\cZ}$ of $\sL^\cZ$.  
Thus to prove that $\VV$ sends $I^{(Q, G)}(\bt, z) $ to $\widetilde{\sL^\cZ}$, it suffices to show that the following diagram commutes when applied to $z \frac{\partial}{\partial t} \left( I^{\cX}(\bm{t}, z) \right)$.
\[
\begin{tikzcd}
  \sL^{\cX} \arrow{r}{\UU} \arrow{d}{ \lim_{\lambda \mapsto 0} \Delta^\circ} & \widetilde{\sL^\cY} \arrow{d}{ \lim_{\lambda \mapsto 0} i^* \circ \Delta^\diamond} \\
  \sL^{(Q,G)} \arrow[dashed]{r}{\VV}
  & \widetilde{\sL^\cZ} . 
\end{tikzcd}
\]

For $g$ such that $N_g > 0$, terms of $z \frac{\partial}{\partial t} \left( I^{\cX}(\bm{t}, z) \right)$ supported in $H^*(\cX_g)$ are in the kernel of the left hand map.  We need therefore to check that these terms are also in the kernel of the composition of the top map with the right hand map.
By the above computation, for $g$ such that $N_g > 0$, the only part of the $\tilde \ii_g$-coefficient of $z \frac{\partial}{\partial t} \widetilde{ I^{\cY}(\bt, z)}$ which survives in the non-equivariant limit is a $\CC((z^{-1}))$-multiple of $\tilde \ii_g H^{N_g-1}$.  This class is in the kernel of the map $i^*: H^*_{CR}([\PP(c_1, \ldots, c_N)/\bar G]) \to H^*_{CR}(\cZ)$.  This implies that the above diagram commutes when applied to $z \frac{\partial}{\partial t} \left( I^{\cX}(\bm{t}, z) \right)$ which proves the claim.

\end{proof}

We arrive at the following.

\begin{theorem}[rCTC implies LG/CY] \label{t:rctcilgcy} Given an LG pair $(Q, G)$ as above with $Q$ a Fermat polynomial and $G$ a subgroup of $SL_N(\CC)$, the refined crepant transformation conjecture for 
$\cY \dasharrow \cX$ (Conjecture~\ref{c:rCTC}) implies the LG/CY correspondence (Conjecture~\ref{c:LGCY}).
\end{theorem}

\begin{proof}
Note first that the $J$-function $J^\cX(\bt, -z)$ satisfies the hypothesis of Lemma~\ref{l:good functions}, as was shown in the proof of Theorem~\ref{t:gw/fjrw}.  Therefore $\VV$ maps $J^{(Q, G)}(\bt, -z)$ to  $\sL^\cZ$.  This implies that $\VV(\sL^{(Q, G)}) \subseteq \widetilde{\sL^\cZ}$.  

On the other hand, consider the function 
$J^\cY(\bt, z)$.  By applying Theorem~\ref{t:qsd} to the particular specializations $\bs$ and $\bs '$ of Section~\ref{sss:lici}, we deduce that
there exists a choice of $\bt '$ such that 
$$\Delta^\diamond \left( z \frac{\partial}{\partial t^{e_{\CC^*}(-K)}} J^\cY(\bt ', z) \right) =  J^{-K, \bs}(\bt, z).$$ 
  Standard argument (see e.g. \cite{Co}, Theorem 1.1) shows that the right hand side has a well defined non-equivariant limit.  Since the map $\Delta^\diamond$ involves division by $e_{\CC^*}(-K)$, we conclude that $z \frac{\partial}{\partial t^{e_{\CC^*}(-K)}} J^\cY(\bt ', z)$
  may be written in the form $e_{\CC^*}(-K) \widetilde{F}(\bt, z)$ where $\widetilde{F}(\bt, z)$ has a well defined non-equivariant limit.  
  So by Theorem~\ref{t:general2},
$$\lim_{\lambda \mapsto 0} i^* \circ \Delta^\diamond \left(-z \frac{\partial}{\partial t^{e_{\CC^*}(-K)}}J^\cY(\bt ', -z)\right) \in \sL^{\cZ}.$$  
Analyzing the non-negative $z$-coefficients of the right hand side yield that in fact \[\lim_{\lambda \mapsto 0} i^* \circ \Delta^\diamond \left(-z \frac{\partial}{\partial t^{e_{\CC^*}(-K)}}J^\cY(\bt ', -z)\right) = J^\cZ(\bt, -z).\]
%
Note furthermore that the non-equivariant limit of $z \frac{\partial}{\partial t^{e_{\CC^*}(-K)}}J^\cY(\bt ', z)$ is contained in the span of classes of compact support, thus 
$$\lim_{\lambda \mapsto 0} \UU^{-1}\left(z \frac{\partial}{\partial t^{e_{\CC^*}(E)}}J^\cY(\bt ', z)\right)$$ 
lies in compact support by assumption (1) of Conjecture~\ref{c:rCTC}.  By Theorem~\ref{t:general1} we have that 
$$\lim_{\lambda \mapsto 0} \Delta^\circ \circ \UU^{-1}\left( \widetilde{ z \frac{\partial}{\partial t^{e_{\CC^*}(E)}}J^\cY(\bt ', z)}\right) \in \sL^{(Q,G)},$$ 
where 
$\widetilde{(-)}$ denotes analytic continuation.

Next, note that in the non-equivariant limit of $\widetilde{ z \frac{\partial}{\partial t^{e_{\CC^*}(E)}}J^\cY(\bt ', z)}$ , the coefficients of $H^{N_g-1}\tilde \ii_g$ in $\widetilde{F}(\bt, z)$ do not contribute.  In other words the diagram
\[
\begin{tikzcd}
  \sL^{\cX}  \arrow{d}{ \lim_{\lambda \mapsto 0} \Delta^\circ} & \arrow{l}{\UU^{-1}} \widetilde {\sL^\cY} \arrow{d}{ \lim_{\lambda \mapsto 0} i^* \circ \Delta^\diamond} \\
  \sL^{(Q,G)}
  & \arrow[dashed]{l}{\VV^{-1}} \widetilde{ \sL^\cZ} . 
\end{tikzcd}
\]
commutes when applied to
 $\widetilde{ z \frac{\partial}{\partial t^{e_{\CC^*}(E)}}J^\cY(\bt ', z)}$.  Therefore $\VV^{-1} \left(\widetilde{J^\cZ(\bt, z)}\right)$ lies in $\sL^{(Q, G)}$.  Thus $\VV(\sL^{(Q, G)} )\supseteq \widetilde{\sL^\cZ}$.

\end{proof}

\section{A proof of the LG/CY correspondence}\label{s:app}
In this section we give a proof of a weak form of the crepant transformation conjecture, which was essentially known already to the experts, and use it in combination with the results of the previous section to deduce the LG/CY correspondence for Fermat polynomials. 

\subsection{The crepant transformation conjecture}

Let the setup be as in Section~\ref{s:ctclgcy}.
The crepant transformation conjecture states that there exists a symplectic transformation $\UU$ which sends $\sL^{\cX}$ to the analytic continuation of $\sL^{\cY}$, thus identifying the two cones.  We will prove a slightly weaker version of this:  we construct two functions $I^{\cX}$ and $I^{\cY}$ which lie on $\sL^{\cX}$ and $\sL^{\cY}$ respectively, and show they are related by analytic continuation and symplectic transformation.  

\begin{theorem}
\label{t:CRCsympl}  There is an explicit linear symplectic transformation $\UU: \sV^{\cY} \to \sV^{\cX}$ which identifies the $I$-function $I^{\cX}$ with the analytic continuation of $I^{\cY}$.
Furthermore, the transformation $\UU$  is induced by a Fourier--Mukai transformation in the sense of Theorem~4.26 of \cite{CIJ}.
\end{theorem}

\begin{remark}
 A general proof of the crepant transformation conjecture, which encompasses Theorem~\ref{t:CRCsympl} is also given in \cite{CIJ}.  To apply Theorem~\ref{t:CRCsympl} towards the LG/CY correspondence, we require specific properties of the functions $I^\cX$ and $I^\cY$ as well as $\UU$, as given in Conjecture~\ref{c:rCTC}.  These properties are not readily apparent in \cite{CIJ}.  Thus we have chosen to explicitly compute $I^\cX$, $I^\cY$, and $\UU$ below.
\end{remark}

\begin{remark}
Because we use so-called \emph{small $I$-functions}, the above theorem does not quite imply the full correspondence between Lagrangian cones.  To recover the full statement one would need to construct \emph{big $I$-functions} as in \cite{CFK} which determine the entire Lagrangian cones.  For the purposes of this paper we content ourselves with the restricted statement.
\end{remark}
%
We first calculate the respective $I$-functions.   

\subsubsection{Setting notation}
Recall that we have a natural choice of basis for the equivariant cohomology of $\cX$ given by $\{\ii_g\}_g \in G$ where $\ii_g$ is the fundamental class of $\cX_g$.  In a slight abuse of notation, we also use $\ii_g$ to denote the corresponding class in $H^*_{CR}(BG)$.  Let $t^g$ denote the dual coordinate to $\ii_g$.  As before we distinguish the dual coordinate to $\ii_\jj$, denoting it as simply $t$.  This will be the analytic continuation coordinate in Theorem~\ref{t:CRCsympl}.  

\begin{notation}\label{n:S}
We let $\{g_s\}_{s \in S}$ denote the set of elements of $G$ which fix at least one coordinate of $\CC^N$ $(N_g > 0)$.  
\end{notation}
For notational convenience we will write $I \cY = \coprod_{g \in G} \cY_g$, with the understanding that $\cY_g$ is empty unless $g \in \{g_s\}_{s \in S}$. 
We also let $t^g$ denote the dual coordinate of $\tilde \ii_g$ for $g \in G$, and let $q$ denote the exponential of the dual coordinate to $H$.  
\subsubsection{The $I$-function of $\cX$}

We consider the $J$-function of $BG$, where the domain has been restricted to the span of $\{\ii_\jj\} \cup \{\ii_{g_s}\}_{s \in S}$.  By Lemmas~\ref{l:2.1} and~\ref{l:2.2}, this coincides with a restriction of $J^{0, \bo}(\bt, z)$ from Lemma~\ref{l:jc0}:
\begin{align*}
J^{BG}(t, \bt, z) 
& = 
z\sum_{\bv{k}\in(\ZZ_{\geq 0})^{S}}\prod_{s\in S}\frac{(t^{g_s})^{k_s}}{z^{k_s}k_s!}
\sum_{k_0\geq 0}\frac{t^{k_0}}{z^{k_0}k_0!}\ii_{\jj^{k_0}\prod_s g_s^{k_s}}. 
\end{align*}

Using the twisted theory technology, one may alter $J^{BG}(t, \bt, z)$ by a \emph{hypergeometric modification} (see \cite{CCIT}) to obtain a function $I^{\cX}(t, \bt, z)$ which generates $\sL^\cX$ in the sense of \eqref{e:generic}.
Let $a(\bv{k})^j=\sum_{s}k_s m_j(g_s)$.  Define the modification factor
\[
M(k_0,\bv{k}) := \prod_{j=1}^N \prod_{l=0}^{\lfloor k_0c_j/d+a(\bv{k})^j \rfloor-1}\Big(-c_j\lambda-(\langle k_0 c_j/d+a(\bv{k})^j  \rangle +l)z\Big)
\]
where $\langle -\rangle$ denotes the fractional part. Then $I^{\cX}(t, \bt, z)$ is defined as 
\begin{equation}\label{e:IX}
I^{\cX}(t, \bt,z)= z t^{d\lambda/z}
\sum_{\bv{k}\in(\ZZ_{\geq 0})^{S}}\prod_{s\in S}\frac{(t^{g_s})^{k_s}}{z^{k_s}k_s!}\sum_{k_0\geq 0}\frac{M(k_0,\bv{k})t^{k_0}}{z^{k_0}k_0!}\ii_{\jj^{k_0}\prod_s g_s^{k_s}}.
\end{equation}
The above modification factor is explained in \cite{CCIT}, where it is proven that $I^{\cX}(t, \bt, z)$ lies on $\sL^\cX$.

\begin{lemma}[Corollary 5.1 \cite{CCIT}]\label{l:LX}
The function $I^{\cX}(t, \bt, z)$ lies on the Lagrangian cone $\sL^\cX$.\footnote{The $I$-function above is often commonly written without the factor of $t^{d\lambda/z}$.  However due to the string equation, multiplication by this factor preserves the cone $\sL^\cX$.  See the remark after Corollary~1 in \cite{CG} for more details.} 

\end{lemma}

\subsubsection{The $I$-function of $\cY$}
An $I$-function for projective toric stacks is given in \cite{CCIT2}.  For the case of $[\PP(G)]$ one obtains:
\[
\begin{split}
I^{[\PP(G)]}(q,\bt,z)=zq^{H/z} &\sum_{\bv{k}\in(\ZZ_{\geq 0})^S}\prod_{s\in S}\frac{(t^{g_s})^{k_s}}{z^{k_s}k_s!}\sum_{k_0\geq 0}q^{k_0/d} \\
&\prod_{j=1}^N\prod_{\substack{0<l\leq k_0c_j/d-a(\bv{k})^j\\ \langle l\rangle = \langle k_0c_j/d-a(\bv{k})^j\rangle}}\frac{1}{(c_jH+lz)}\tilde \ii_{\jj^{-k_0}\prod_s g_s^{k_s}}.
\end{split}
\]
We alter this by another hypergeometric modification to obtain an $I$-function for $\cY$:
\begin{equation*}
\begin{split}
I^{\cY}(q,\bt,z)=zq^{H/z}&\sum_{\bv{k}\in(\ZZ_{\geq 0})^S} \prod_{s\in S}\frac{(t^{g_s})^{k_s}}{z^{k_s}k_s!} \\
&\qquad \sum_{k_0\geq 0}q^{k_0/d}\frac{\prod_{l=0}^{k_0-1}(-d(H+\lambda)-lz)}{\prod_{j=1}^N\prod_{\substack{0<l\leq k_0c_j/d-a(\bv{k})^j\\ \langle l\rangle = \langle k_0c_j/d-a(\bv{k})^j\rangle}}(c_jH+lz)}\tilde \ii_{\jj^{-k_0}\prod_s g_s^{k_s}}.
\end{split}
\end{equation*}

\begin{lemma}[Corollary 5.1 \cite{CCIT}, Theorem 21 \cite{CCIT3}]\label{l:LY}
The function $I^{\cY}(t, \bt, z)$ lies on the Lagrangian cone $\sL^\cY$.
\end{lemma}

Next we will show that the above $I$-functions coincide after analytic continuation and symplectic transformation.

\subsubsection{$\hat{\Gamma}$-classes:}
In order to facilitate our analytic continuation, we will write the $I$-functions in a different form, motivated by Iritani's integral structure for quantum cohomology \cite{Iri}. To do this we define the so-called \emph{$\hat{\Gamma}$-classes}, coming from K--theory. 

For $\cX$, the $\hat{\Gamma}$-class is defined by 
\[\hat{\Gamma}(\cX):=\bigoplus_g\prod_{j=1}^N \Gamma(1-m_j(g)-c_j\lambda).\]
Note that if $g=\jj^{k_0}\prod_{s\in S} g_s^{k_s}$, then $m_j(g)=\langle k_0 c_j/d+a(\bv k)^j\rangle$. 
 
 \begin{notation}
 Define the operator $Gr: H^*_{CR}(\cX) \to H^*_{CR}(\cX)$ by $\alpha \mapsto \frac{\deg(\alpha)}{2}\alpha$ for $\alpha$ of pure degree in $H^*_{CR}(\cX)$. Here by degree we mean the real Chen--Ruan degree.
 
In addition, for $\alpha$ a cohomology class of pure degree in $H^*(I\cX)$ supported on a single connected component, define the function $\deg_0(\alpha)$ to be the \emph{untwisted} degree of $\alpha$ in $H^*(I\cX)$.  
\end{notation}

Consider the modification factor for $I^{\cX}(t,\bt,z)$ again. Using the relation $(x)(x-z)\dots(x-(n-1)z)=z^n\frac{\Gamma(1+x/z)}{\Gamma(1-n+x/z)}$, we obtain
\begin{align*}M(k_0,\bv{k}) &= z^{\sum_j\lfloor k_0c_j/d+a(\bv{k})^j\rfloor}\prod_{j=1}^N \frac{\Gamma(1-c_j\tfrac{\lambda}{z}- \langle k_0c_j/d+a(\bv{k})^j\rangle)}{\Gamma(1-c_j\tfrac{\lambda}{z}- k_0c_j/d-a(\bv{k})^j)}.\end{align*}

Via the above expression and the equality 
\begin{align*}
 k_0 + \sum_{s\in S} \age(g_s) k_s&= \sum_{j=1}^N k_0c_j/d+a(\bv{k})^j \\
 &= \sum_{j=1}^N \lfloor k_0c_j/d+a(\bv{k})^j \rfloor
+ \sum_{j=1}^N \langle k_0c_j/d+a(\bv{k})^j \rangle,
\end{align*}
$I^{\cX}(t,\bt,z)$ simplifies to 
\begin{align*}
I^{\cX}(t, \bt,z) &= zt^{d\lambda/z}\sum_{\bv{k}\in(\ZZ_{\geq 0})^S}\prod_{s\in S}\frac{(t^{g_s})^{k_s}z^{(\age(g_s)-1)k_s}}{k_s!}\sum_{k_0\geq 0}\frac{t^{k_0}}{z^{\sum_j\langle k_0c_j/d+a(\bv{k})^j\rangle}k_0!}\\
  &\qquad \cdot \prod_{j=1}^N \frac{\Gamma(1-c_j\tfrac{\lambda}{z}- \langle k_0c_j/d+a(\bv{k})^j\rangle)}{\Gamma(1-c_j\tfrac{\lambda}{z}- k_0c_j/d-a(\bv{k})^j))}\ii_{\jj^{k_0}\prod_s g_s^{k_s}}\\
	&=z^{1-\Gr}\hat{\Gamma}(\cX)(2\pi i)^{\deg_0/2}H(t,\bt,z)	
\end{align*}
where
\begin{align*}
H(t,\bt,z)&=t^{d\lambda/(2\pi i)}\sum_{\bv{k}\in(\ZZ_{\geq 0})^S}\prod_{s\in S}\frac{(t^{g_s})^{k_s}z^{(\age(g_s)-1)k_s}}{k_s!}\\
  &\qquad\cdot\sum_{k_0\geq 0}\frac{t^{k_0}}{k_0!}\prod_{j=1}^N \frac{1}{\Gamma(1-c_j\tfrac{\lambda}{2\pi i}- k_0c_j/d-a(\bv{k})^j)}\ii_{\jj^{k_0}\prod_s g_s^{k_s}}\\
      &=t^{d\lambda/(2\pi i)}\sum_{\bv{k}\in(\ZZ_{\geq 0})^S}\prod_{s\in S}\frac{(t^{g_s})^{k_s}z^{(\age(g_s)-1)k_s}}{k_s!}\\
  &\cdot \sum_{0\leq m< d}\sum_{k\geq 0}\frac{t^{m+dk}}{(m+dk)!}\prod_{j=1}^N \frac{1}{\Gamma(1-c_j\tfrac{\lambda}{2\pi i}- kc_j- mc_j/d-a(\bv{k})^j)}\ii_{\jj^{m}\prod_s g_s^{k_s}}.
\end{align*}

For $\cY$ we define the $\hat{\Gamma}$-class as the transformation on $H^*_{CR}(\cY)$
\[
\hat{\Gamma}(\cY)=\bigoplus_g\Gamma(1-d(H+\lambda))\prod_{j=1}^N\Gamma(1-m_j(g)+c_jH).
\]

By a similar argument to the previous case, we can also rewrite $I^{\cY}$ in terms of Gamma functions.
\begin{align*}
I^{\cY}(q,\bt,z) &=zq^{H/z}\sum_{\bv{k}\in(\ZZ_{\geq 0})^S}\prod_{s\in S}\frac{(t^{g_s})^{k_s}z^{(\age(g_s)-1)k_s}}{k_s!} \\
 &\qquad \qquad \sum_{k_0\geq 0}\frac{q^{k_0/d}}{z^{\sum_j\langle k_0c_j/d-a(\bv k)^j\rangle}}\frac{\Gamma(1-\tfrac{d(\lambda+H)}{z})}{\Gamma(1-k_0-\tfrac{d(\lambda+H)}{z})}\\
	&\cdot \qquad \qquad \quad\cdot \prod_{j=1}^N\frac{\Gamma(1+c_jH/z-\langle -k_0c_j/d+a(\bv{k})^j \rangle)}{\Gamma(1+c_jH/z+k_0c_j/d-a(\bv{k})^j)}\tilde \ii_{\jj^{-k_0}\prod_s g_s^{k_s}}\\
	&=z^{1-\Gr}\hat{\Gamma}(\cY)(2\pi i)^{\deg_0/2}H^{\cY}(q,\bt,z),
\end{align*}
where
\begin{align*}
H^{\cY}(q,\bt,z) &=q^{H/2\pi i}\sum_{\bv{k}\in(\ZZ_{\geq 0})^S}\prod_{s\in S}\frac{(t^{g_s})^{k_s}z^{(\age(g_s)-1)k_s}}{k_s!} \\
 &\qquad \sum_{k_0\geq 0}q^{k_0/d}\frac{\Gamma(k_0+\tfrac{d(\lambda+H)}{2\pi i})\sin(\pi(k_0+\tfrac{d(\lambda+H)}{2\pi i}))}{\pi \prod_{j=1}^N\Gamma(1+\tfrac{c_jH}{2\pi i}+k_0c_j/d - a(\bv{k})^j)}\tilde \ii_{\jj^{-k_0}\prod_s g_s^{k_s}}\\
	&=q^{H/2\pi i}\sum_{\bv{k}\in(\ZZ_{\geq 0})^S}\prod_{s\in S}\frac{(t^{g_s})^{k_s}z^{(\age(g_s)-1)k_s}}{k_s!}\sum_{0\leq b<d}q^{b/d}(-1)^b \frac{ \sin(\pi\tfrac{d(\lambda+H)}{2\pi i})}{\pi} \\
	&\qquad \cdot\sum_{k\geq 0}q^k(-1)^{dk}\frac{\Gamma(b+dk+\tfrac{d(\lambda+H)}{2\pi i})}{\prod_{j=1}^N\Gamma(1+\tfrac{c_jH}{2\pi i}-a(\bv{k})^j+bc_j/d+k)}\tilde \ii_{\jj^{-b}\prod_s g_s^{k_s}}.
\end{align*}
In the last equality, we have made the substitution $k_0=b+dk$ for $0\leq b < d$. 

In order to show that these functions agree, we must analytically continue $H^{\cY}(q,\bt,z)$. We will use the Mellin--Barnes method.  We may rewrite the above expression using residues:
\begin{align*}
H^{\cY}(q,\bt,z) &= q^{H/2\pi i}\sum_{\bv{k}\in(\ZZ_{\geq 0})^S}\prod_{s \in S} \frac{(t^{g_s})^{k_s}z^{(\age(g_s)-1)k_s}}{k_s!} \\
& \qquad \qquad \cdot\sum_{0\leq b<d}(-1)^{b} q^{b/d}\frac{\sin(\tfrac{d(\lambda+H)}{2i})}{\pi}\tilde \ii_{\jj^{-b}\prod_s g_s^{k_s}}\\
	&\quad-\int_C\frac{e^{-\pi i ds}q^s}{e^{-2\pi i s}-1}\frac{\Gamma(ds+b+\tfrac{d(\lambda+H)}{2\pi i})}{\prod_j\Gamma(1+c_jH/(2\pi i)-a(\bv{k})^j+bc_j/d+c_js)}ds.
\end{align*}
Here $C$ is a contour going clockwise along the imaginary axis, enclosing the non--negative integers to the right, and enclosing no other poles. 

Closing the contour to the left yields the analytic continuation. There are poles at the negative integers due to the exponential, but these vanish due to factors of $H$. Indeed whenever we are supported on $\cY_{\jj^{-b}\prod_s g_s^{k_s}}$, the residue at a negative integer will contribute a factor of $c_jH/z$ for each $j$ which is fixed by $\jj^{-b}\prod_s g_s^{k_s}$. There are $  N_{\jj^{-b}\prod_s g_s^{k_s}}=\dim( \cY_{\jj^{-b}\prod_s g_s^{k_s}}) + 1$ such factors. The other poles are from the Gamma function in the numerator, and occur at 
\[
s=-(H+\lambda)/(2\pi i)-b/d-m/d \quad \text{for} \, m\geq 0. 
\]
The residue of the Gamma function here is 
\[
\Res_{s=-(H+\lambda)/(2\pi i)-b/d-k/d)}\Gamma(ds+b+\tfrac{d(\lambda+H)}{2\pi i})=\frac{(-1)^k}{d\cdot k!}.
\]
We obtain as the analytic continuation ${H^{\cY}}'(q,\bt,z)$: 
\begin{align*}
{H^{\cY}}' &(q,\bt,z) = 2 \pi i q^{H/2\pi i} \sum_{\bv{k}}\prod_{s \in S}\frac{(t^{g_s})^{k_s}z^{(\age(g_s)-1)k_s}}{k_s!} \\
&\quad \cdot \sum_{0\leq b< d}(-1)^{b} q^{b/d}\frac{\sin(d(\lambda+H)/2i)}{\pi}\tilde \ii_{\jj^{-b}\prod_s g_s^{k_s}}\\
	&\quad \cdot\sum_{m\geq 0}\frac{e^{\pi i(b+m)}e^{d(H+\lambda)/2}}{e^{2\pi i(b+m)/d}e^{(\lambda+H)}-1}\cdot\frac{(-1)^m}{d\cdot m!} \frac{q^{-(b+m)/d-(\lambda+H)/(2\pi i)}} {\prod_j\Gamma(1-c_j\lambda/z-mc_j/d-a(\bv{k})^j)}\\
	&=q^{-\lambda/(2\pi i)}\sum_{\bv{k}}\prod_{s \in S}\frac{(t^{g_s})^{k_s}z^{(\age(g_s)-1)k_s}}{k_s!} \\
	&\quad \cdot \sum_{m\geq 0}\frac{q^{-m/d}}{m!\prod_j\Gamma(1-c_j\lambda/(2\pi i)- mc_j/d-a(\bv{k})^j)}\\
	&\qquad \cdot\sum_{0\leq b<d}\frac{e^{d(\lambda+H)/2}(e^{d(\lambda+H)/2}-e^{-d(\lambda+H)/2})}{d(e^{(\lambda+H)}\xi^{b+m}-1)}\tilde \ii_{\jj^{-b}\prod_s g_s^{k_s}}\\
	&=t^{d\lambda/(2\pi i)}\sum_{\bv{k}}\prod_{s \in S}\frac{(t^{g_s})^{k_s}z^{(\age(g_s)-1)k_s}}{k_s!} \\
	&\quad \cdot \sum_{m\geq 0}\frac{t^{m}}{m!\prod_j\Gamma(1-c_j\lambda/(2\pi i)- mc_j/d-a(\bv{k})^j)}\\
	&\qquad \cdot\sum_{0\leq b<d}\frac{e^{d(\lambda+H)}-1}{d(e^{(\lambda+H)}\xi^{b+m}-1)}\tilde \ii_{\jj^{-b}\prod_s g_s^{k_s}}	
\end{align*}
where we have made the substitution $q = t^{-d}$.

\subsubsection{The transformation}
Consider the transformation 
$$\overline{\UU}: H^*_{CR}(\cX) \to H^*_{CR}(\cY)$$ 
given by
\begin{align}\label{e:olu}
\overline{\UU}:
\ii_{g
} \mapsto \sum_{0\leq b<d}\frac{e^{d(\lambda+H)}-1}{d(e^{(\lambda+H)}\xi^{b}-1)}\tilde \ii_{g\jj^{-b}}.
\end{align}

A simple check shows that 
\begin{align*}
\overline{\UU}(
\ii_{g\jj^m}) = \sum_{0\leq b<d}\frac{e^{d(\lambda+H)}-1}{d(e^{(\lambda+H)}\xi^{b+m}-1)}\tilde \ii_{g\jj^{-b}},
\end{align*}
from which one sees that $\overline{\UU}(H^{\cX}(t,\bt,z)) = {H^{\cY}}'(t,\bt,z)$.

%
%
%

%

\begin{definition} Define the linear transformation
$
\UU: \sV^{\cX}\to \sV^{\cY}
$ by
\[\UU:=z^{-\Gr}\hat{\Gamma}(\cY)(2\pi i)^{\deg_0/2}\overline{\UU} (2\pi i)^{-\deg_0/2}\hat{\Gamma}(\cX)^{-1}z^{\Gr}.\]
\end{definition}

The linear transformation $\UU$ gives the desired identification between $\sL^\cX$ and $\sL^\cY$ from Theorem~\ref{t:CRCsympl}.

\begin{proof}[Proof of Theorem~\ref{t:CRCsympl}]
Lemmas~\ref{l:LX} and~\ref{l:LY} provide the $I$-functions for the respective theories.
It is clear by construction that $\UU(I^{\cX}(t,\bt,z))=\widetilde{{I^{\cY}}(t,\bt,z)}$, where $\widetilde{{I^{\cY}}(t,\bt,z)}$ is the analytic continuation of ${I^{\cY}}(q,\bt,z)$.  One can check that the function  $\UU$ defined above agrees with that given in \cite{CIJ}.  It is proven in \cite{CIJ} that $\UU$ is symplectic and is compatible with a Fourier--Mukai transform.
%
%
\end{proof}

The explicit description of  $\UU$ allows us to immediately deduce the following.

\begin{proposition}\label{p:restrict}The transformation $\UU$ defined above satisfies the conditions of Conjecture~\ref{c:rCTC}.
%
%
%
%
\end{proposition}

\begin{proof}
%

From the explicit expression for $\overline{\UU}$, it is clear that $\UU$ has a well-defined non-equivariant limit.  That this limit induces an isomorphism on the restriction to compactly supported classes follows from the fact that it is induced by a Fourier--Mukai transformation \cite{CIJ}.  
%
%
%

To check condition \eqref{i:c2} of the conjecture, note that because $z^{-\Gr}\hat{\Gamma}( -) (2\pi i)^{\deg_0}$ acts diagonally on both cohomologies, it is enough to show that the image of $\overline{\UU}$ satisfies condition \eqref{i:c2}.

By the formula \eqref{e:olu} for $\overline{\UU}$, we see that the coefficient of $\tilde \ii_{g \jj^{-b}}$ in $\overline{\UU}(\ii_g)$ is in the $\CC[\lambda]$-span of $(\lambda + H)$ unless $b = 0$.  When $b = 0$ the coefficient may be expanded as
\[\frac{1}{d}\left(e^{(d-1)(\lambda +H)} + e^{(d-2)(\lambda +H)} + \cdots + e^{(\lambda +H)} + 1
\right).\]
This proves the claim.

%

\end{proof}

%
%
%
%
%
%
%
%
%
%
%
%
%
%

\subsection{The LG/CY correspondence}
Let the setup be as in Section~\ref{s:ctclgcy}.
In this section we use Theorem~\ref{t:CRCsympl} and the results of Section~\ref{s:ctclgcy} to prove the following:
\begin{theorem}[The LG/CY correspondence for $(Q, G)$]\label{LGCYsympl}  There exist $I$-functions $I^{(Q, G)}$ and $I^{\cZ}$ lying on $\sL^{(Q, G)}$ and $\sL^{\cZ}$ respectively, and a linear symplectic transformation $\VV: \sV^{(Q, G)} \to \sV^{\cZ}$ which identifies (up to a change of variables) the $I$-function $I^{(Q, G)}$ with the analytic continuation of $I^{\cZ}$.
\end{theorem}

\begin{proof}
The proof amounts to checking that the conditions of Conjecture~\ref{c:rCTC} are satisfied 
by our symplectic transformation $\UU$.  This follows immediately from Proposition~\ref{p:restrict}.  Therefore the transformation $\VV$ as defined in Lemma~\ref{l:5.6} is symplectic.

Consider the function $I^\cX(t, \bt, z)$ from~\eqref{e:IX}.  
Note that in the formula for $I^\cX(t, \bt, z)$, the modification factor $M(k_0, \bv{k})$ contains a factor of $-c_j \lambda$ whenever $k_0 > 0$ and 
$\jj^{k_0}\prod_s g_s^{k_s}$ fixes the $j$th coordinate.  Thus, for $\ii_g$ supported on a non-compact set (i.e. for $g$ such that $N_g > 0$), the coefficient of $\ii_g$ in $ \frac{\partial}{\partial t} \left( I^{\cX}(t, \bt, z) \right)$ is divisible by $\lambda$ and therefore vanishes in the non-equivariant limit.  Therefore $I^\cX(t, \bt, z)$ satisfies the assumptions of Lemma~\ref{l:good functions}.  Defining $I^{(Q,G)}$ as in Lemma~\ref{l:good functions}, we conclude that $\VV$ maps $I^{(Q,G)}$ to a function $\widetilde{I^\cZ}$ lying in the analytic continuation of $\sL^\cZ$.
\end{proof}

\bibliographystyle{plain}
\bibliography{references}

\end{document}